\documentclass[a4]{amsart}

\input xypic 
\input xy 
\xyoption{all} 
\usepackage{amssymb} 
\usepackage{bbm}
\usepackage{tikz-cd}

\usepackage{float}
\usepackage{graphicx}
\usepackage{accents}
\usepackage[bookmarksnumbered, bookmarksopen, 
colorlinks,citecolor=blue,linkcolor=blue,backref]{hyperref}

\newtheorem{theorem}{Theorem}[section]
\newtheorem{proposition}[theorem]{Proposition}
\newtheorem{lemma}[theorem]{Lemma}
\newtheorem{corollary}[theorem]{Corollary}

\theoremstyle{definition}

\newtheorem{remark}[theorem]{Remark}

\newcounter{RomanNumber}

\newcommand{\testcap}{\mathrel{\vcenter{\offinterlineskip
\hbox{$\cap$}\vskip-1.8ex\hbox{$\kern0.2em | \kern0.15em$}}}}

\newcounter{bean}

\newcommand{\nameddright}[5]{\ensuremath{#1\stackrel{#2}
 {\larrow}#3\stackrel{#4}{\larrow}#5}}

\newcommand{\larrow}{\relbar\!\!\relbar\!\!\rightarrow}
\newcommand{\llarrow}{\relbar\!\!\relbar\!\!\larrow}
\newcommand{\lllarrow}{\relbar\!\!\relbar\!\!\llarrow}

\newcommand{\llnameddright}[5]{\ensuremath{#1\stackrel{#2}
 {\llarrow}#3\stackrel{#4}{\llarrow}#5}}

\newcommand{\qqed}{\hfill\Box}

\begin{document}


\title{Homotopy of Simply Connected Complexes with a Spherical Pair} 

\author{Ruizhi Huang} 
\address{State Key Laboratory of Mathematical Sciences \& Institute of Mathematics, Academy of Mathematics and Systems Science, 
   Chinese Academy of Sciences, Beijing 100190, China} 
\email{huangrz@amss.ac.cn} 
   \urladdr{https://sites.google.com/site/hrzsea/}



\begin{abstract} 
We establish a loop space decomposition for certain $CW$-complexes with a single top cell in the presence of a spherical pair, thereby generalizing several known decompositions of Poincar\'{e} duality complexes in which a loop of a product of spheres appears as a direct summand.

This decomposition is further applied to derive results on local hyperbolicity, on inertness and non-inertness, on the gaps between rational inertness and local or integral inertness, and on the homotopy theory of smooth manifolds with transversally embedded spheres. 

In particular, in every dimension greater than three, there exist infinitely many finite $CW$-complexes, pairwise non-homotopy-equivalent, whose loop spaces retract off the loops of their lower skeletons rationally but not locally, and whose top cell attachments produce infinitely many new torsion homotopy groups with exponentially growing ranks.
\end{abstract}

\maketitle


\section{Introduction} 

The study of the homotopy theory of $CW$-complexes through their cohomological data is a classical problem in algebraic topology. In particular, significant recent progress has been made in understanding the homotopy theory of Poincar\'{e} duality complexes. For instance, loop space decompositions of $(n-1)$-connected $2n$-manifolds were investigated by Beben-Theriault \cite{BT14} and Sa. Basu-So. Basu \cite{BB18}, while those of $(n-1)$-connected $(2n+1)$-manifolds were studied by Beben-Wu \cite{BW15}, Basu \cite{Bas19}, Beben-Theriault \cite{BT22}, and Huang-Theriault \cite{HT22}. Moreover, particular $6$-manifolds were analyzed by Huang \cite{Hua23}, and a unified loop space decomposition for Poincar\'{e} duality complexes was established by Stanton-Theriault \cite{ST25} under reasonable conditions.

From the perspective of homotopy theory, many known cases of loop space decompositions of Poincar\'{e} duality complexes are derived from a good pair of cohomology classes that are Poincar\'{e} dual to one another. Such a pair can often be realized as a twisted or untwisted product of spheres, which retracts off the Poincar\'{e} duality complex after looping. From the perspective of geometric topology, Haibao Duan proposed studying the homotopy theory of smooth manifolds through the intersection of two embedded spheres based on a result of Huang \cite{Hua24}. 

In this paper, we develop these ideas in a broader setting and establish new loop space decompositions for more general $CW$-complexes. Specifically, we prove a loop space decomposition for the so-called capped complexes in the presence of a spherical pair, and apply it to obtain a loop space decomposition for manifolds with two transversally embedded spheres. Furthermore, these loop space decompositions are used to derive results on local hyperbolicity, inertness and non-inertness, as well as on the gaps between rational inertness and local or integral inertness.

To present our main results, we first introduce the necessary notions and conventions. 
A simply connected $n$-dimensional $CW$-complex $Y$ is called a {\it capped complex} if it has a single $n$-cell and satisfies $H_{n}(Y;\mathbb{Z})\cong \mathbb{Z}$, or equivalently, if there exists a homotopy cofibration
$
S^{n-1}\stackrel{g}{\larrow} \overline{Y}\stackrel{}{\larrow} Y,
$
where $\overline{Y}$ denotes the $(n-1)$-skeleton of $Y$, and the attaching map $g$ induces the trivial homomorphism in homology.
For a capped complex $Y$, a generator of $H_{n}(Y;\mathbb{Z})\cong \mathbb{Z}$ is called the {\it fundamental class} of $Y$ and is denoted by $[Y]$.
Note that if a simply connected finite-dimensional $CW$-complex has its top nontrivial homology group isomorphic to the infinite cyclic group, then it is homotopy equivalent to a capped complex. In particular, every simply connected Poincar\'{e} duality complex is homotopy equivalent to a capped complex.

For pointed spaces $A$ and $B$, 
the \emph{right half-smash} is defined as the quotient space 
$A\rtimes B=(A\times B)/(\{\ast\} \times B)$. 
The {\it Moore space $P^n(k)$} is defined as the homotopy cofibre of the degree $k$ self-map of the sphere $S^{n-1}$. 
Note that $P^n(k)$ is contractible when $k=\pm 1$, and $P^n(k)\simeq S^{n-1}\vee S^{n}$ when $k=0$.
Denote by $[i_1, i_2]$ the Whitehead product of the canonical inclusions $S^m\stackrel{i_1}{\larrow} S^m\vee S^{n-m}$ and $S^{n-m}\stackrel{i_2}{\larrow} S^m\vee S^{n-m}$. It is known that the Whitehead product $[i_1, i_2]$ detects a pair of cohomology classes of $S^m\times S^{n-m}$ that are Poincar\'{e} dual to one another.

\begin{theorem}\label{Zdecom-thm}
Let $X$ be an $n$-dimensional simply connected capped complex determined by a homotopy cofibration
\[
S^{n-1}\stackrel{k\mathfrak{g}}{\larrow} S^{m}\vee S^{n-m}\vee C\stackrel{}{\larrow} X
\]
for a space $C$, $k\in \mathbb{Z}$ and $2\leq m\leq n-2$. Suppose that the component of the map $\mathfrak{g}$ on $S^m\vee S^{n-m}$ is the Whitehead product $[i_1, i_2]$. 
Then there is a homotopy fibration
\[
(P^{n}(k)\vee C)\rtimes \Omega (S^m\times S^{n-m})\stackrel{}{\larrow} X\stackrel{}{\larrow} S^m\times S^{n-m},
\]
which splits after looping to give a homotopy equivalence 
   \[
   \Omega X\simeq\Omega (S^m\times S^{n-m})\times\Omega \big((P^{n}(k)\vee C)\rtimes \Omega (S^m\times S^{n-m})\big). 
 \] 
\end{theorem}

This loop space decomposition also satisfies a naturality property, which is discussed in detail in Remark~\ref{natremark}. Theorem \ref{Zdecom-thm} may be regarded as a generalization of several known decomposition results for Poincar\'{e} duality complexes admitting a good pair of cohomology classes that can be realized as a twisted or untwisted product of spheres \cite{BT14, BW15, BT22}, and it further applies to the study of the homotopy theory of manifolds through intersections of embedded spheres; see Theorem \ref{int-thm} below.

From a homotopy-theoretic perspective, a distinctive feature of Theorem \ref{Zdecom-thm} is the appearance of a summand $P^{n}(k)$ contributing to the homotopy of $X$, which does not arise directly from the cell structure of the lower skeleton $\overline{X}$. This additional piece leads to notable applications, including results on torsion in the homotopy groups of $X$ and the study of the non-inertness property discussed in the sequel.

For the asymptotic behavior of torsion in homotopy groups, Huang-Wu \cite{HW20} introduced the notion of local hyperbolicity as a local analogue of rational hyperbolicity in rational homotopy theory \cite{FHT09}. Given a prime $p$ and a positive integer $r$, a space $Y$ is said to be {\it $\mathbb{Z}/p^r$-hyperbolic} if the number of $\mathbb{Z}/p^r$-summands in $\pi_\ast(Y)$ grows exponentially. This notion has since been extensively studied by Huang-Wu \cite{HW20}, Zhu-Pan \cite{ZP21}, Boyde \cite{Boy22, Boy24}, Huang-Theriault \cite{HT24}, and Boyde-Huang \cite{BH24}. 
Here, Theorem \ref{Zdecom-thm} can be applied to prove the following result on local hyperbolicity. 

\begin{theorem}\label{exp-thm}
Let $X$ be the capped complex in Theorem \ref{Zdecom-thm}. The following hold: 
\begin{itemize}
\item if $k=0$, then $X$ is rationally hyperbolic, and is $\mathbb{Z}/p^r$-hyperbolic for all primes $p$ and all $r\geq 1$; 
\item if $k\neq 0$, then $X$ is $\mathbb{Z}/p^r$- and $\mathbb{Z}/p^{r+1}$-hyperbolic for all odd primes $p$ with $p^r~|~k$;
\item if $k\neq 0$ and $4~|~k$, then $X$ is $\mathbb{Z}/2^r$- and $\mathbb{Z}/2^{r+1}$-hyperbolic for all $2^r~|~k$;
\item if $k\neq 0$ and $2~|~k$ but $4~\nmid k$, then $X$ is $\mathbb{Z}/2$-, $\mathbb{Z}/4$- and $\mathbb{Z}/8$-hyperbolic.
\end{itemize}
\end{theorem}

In addition to the hyperbolicity result, stronger information about the homotopy groups of $X$ can be established. In Lemma \ref{Znoninert-lemma}, we show that, when $k\neq \pm 1$, for the homomorphism
$\pi_\ast(S^{m}\vee S^{n-m}\vee C)\larrow \pi_\ast(X)$
induced by the inclusion of the $(n-1)$-skeleton, the number of the torsion summands in the graded cokernel grows exponentially. In particular, the top cell attachment of $X$ produces infinitely many new torsion homotopy groups whose ranks grow exponentially, thereby leading to the subsequent discussion of the non-inertness property.

Theorem \ref{Zdecom-thm} can be applied to study the so-called inertness property. 
Let $A\stackrel{h}{\longrightarrow} X\stackrel{\varphi}{\longrightarrow} Y$ be a homotopy cofibration. 
The map $h$ is called {\it inert} if $\Omega \varphi$ has a right homotopy inverse. Rational inertness was introduced by F\'{e}lix-Halperin-Thomas \cite{FHT82} and was widely investigated in the context of rational homotopy theory. For instance, a remarkable result of Halperin and Lemaire \cite{HL87} shows that the attaching map for the top cell of any Poincar\'{e} duality complex is rationally inert unless its rational cohomology algebra is generated by a single element. 
In contrast, the local and integral inertness were studied much later in the context of unstable homotopy theory \cite{BT22, Hua24, The24a, The24b}.

However, comparatively little work has been devoted to non-inertness and to the gaps between rational inertness and local or integral inertness. Our next theorems provide new insight in these directions.

 \begin{theorem}\label{Zinert-thm}
 Let $X$ be the capped complex in Theorem \ref{Zdecom-thm}. Then the attaching map for the top cell of $X$ is inert if and only if $k=\pm 1$. 
 \end{theorem}

The inertness part of Theorem \ref{Zinert-thm} can be applied to calculate the loop space algebra of $X$. 
Let $R$ be a commutative ring with unit. Denote by $T(- ; R)$ the free tensor algebra functor over $R$. For a graded $R$-module $V$, denote by $\Sigma^{-1}V$ the graded $R$-module obtained by shifting every element of $V$ down by one degree.

\begin{theorem}\label{loopH-thm}
Let $X$ be the capped complex in Theorem \ref{Zdecom-thm} with $k=\pm 1$ and $C$ being a co-$H$-space. 
Let $S^{n-2}\stackrel{\widetilde{\mathfrak{g}}}{\larrow} \Omega(S^m\vee S^{n-m}\vee C)$ be the adjoint of $\mathfrak{g}$. Then for any commutative ring $R$ with unit such that $H_\ast(C; R)$ is a free $R$-module, there is an algebra isomorphism
\[
H_\ast(\Omega X; R)\cong T(\Sigma^{-1} \widetilde{H}_\ast(S^m\vee S^{n-m}\vee C); R)/ (\text{Im}~(\widetilde{\mathfrak{g}}_\ast)),
\]
where $(\text{Im}~(\widetilde{\mathfrak{g}}_\ast))$ is the two sided ideal generated by $\text{Im}~(\widetilde{\mathfrak{g}}_\ast)$. Moreover, if $C$ is the suspension of a co-$H$-space, then this is an isomorphism of Hopf algebras. 
\end{theorem}

The non-inertness part of Theorem \ref{Zinert-thm} can be applied to construct explicit examples that illustrate the gaps between rational inertness and local or integral inertness.  

\begin{theorem}\label{QZinert-thm}
In every dimension greater than three, there exist infinitely many finite capped complexes, pairwise non-homotopy-equivalent, whose top-cell attaching maps are rationally inert but not integrally inert.

Moreover, in every dimension greater than three, and for each prime $p$, there exist infinitely many finite capped complexes, pairwise non-homotopy-equivalent, whose top-cell attaching maps are rationally inert but not locally inert at $p$.
\end{theorem}

Additionally, by Remark \ref{ex-remark}, the examples constructed for Theorem \ref{QZinert-thm} can be chosen to be both rationally and locally hyperbolic, with their top cell attachments producing infinitely many new torsion homotopy groups whose ranks grow exponentially.

We remark that all the examples constructed for Theorem \ref{QZinert-thm} are not Poincar\'{e} Duality complexes. It remains an interesting open question whether the top-cell attachments of all Poincar\'{e} Duality complexes are integrally inert, except in the case where their rational cohomology algebras are generated by a single element.

We also establish a local analogue of Theorem \ref{Zdecom-thm}. Motivated both by the hypotheses in that theorem and by the context of intersection theory, we introduce the notion of a spherical pair. 
For a space $X$, a pair of cohomology classes $(a, b)\in H^{m}(X;\mathbb{Z})\times  H^{n-m}(X;\mathbb{Z})$ is called a {\it spherical pair} if $a^2=0$ and $b^2=0\in H^\ast(X;\mathbb{Z})$, and there exist maps
\[
s_1: S^m\larrow X \ \ \ {\rm and} \ \ \  s_2: S^{n-m}\larrow X
\]
such that 
$s_1^\ast(a)$ and $s_2^\ast(b)$ are generators of $H^{m}(S^m;\mathbb{Z})$ and $H^{n-m}(S^{n-m};\mathbb{Z})$, respectively. 

\begin{theorem}\label{pdecom-thm}
Let $X$ be an $n$-dimensional simply connected capped complex such that its $(n-1)$-skeleton is a co-$H$-space. 
Suppose that $(a, b)\in H^{m}(X;\mathbb{Z})\times  H^{n-m}(X;\mathbb{Z})$ with $2\leq m\leq n-2$ is a spherical pair of $X$ such that $\langle a\cup b, [X]\rangle\neq 0$, and the attaching map for the top cell of $X$ is 
divisible by $k\in \mathbb{Z}$. Then $k~|~\langle a\cup b, [X]\rangle$, and 
after localization away from any prime $p$ satisfying $2p<{\rm max}(m, n-m)+4$ or $p~|~\frac{\langle a\cup b, [X]\rangle}{k}$, the following hold:
\begin{itemize}
\item there is a homotopy cofibration
\[
S^{n-1} \larrow S^m\vee S^{n-m}\vee C\stackrel{}{\larrow} X
\]
for a simply connected space $C$ of dimension less than $n$;
\item
there is a homotopy fibration
\[
(P^{n}(k)\vee C)\rtimes \Omega (S^m\times S^{n-m})\stackrel{}{\larrow} X\stackrel{}{\larrow} S^m\times S^{n-m};
\]
\item the homotopy fibration splits after looping to give a homotopy equivalence 
   \[
   \Omega X\simeq\Omega (S^m\times S^{n-m})\times\Omega \big((P^{n}(k)\vee C)\rtimes \Omega (S^m\times S^{n-m})\big). 
 \] 
 \end{itemize}
\end{theorem}

In the context of intersection theory, a geometric analogue of Theorem \ref{pdecom-thm} can be formulated as follows. For any homology class $x\in H_\ast(Y)$, denote by $x^\ast\in H^\ast(Y)$ its Kronecker dual. 

\begin{theorem}\label{int-thm}
Let $M$ be an $n$-dimensional simply connected closed smooth manifold such that its $(n-1)$-skeleton is a co-$H$-space. Suppose that there exist two transversally embedded spheres
\[
s_1: S^m\larrow M \ \ \ {\rm and} \ \ \  s_2: S^{n-m}\larrow M
\]
with $2\leq m\leq n-2$. Denote by $a=(s_{1\ast}([S^m]))^\ast$ and $b=(s_{2\ast}([S^{n-m}]))^\ast$. Suppose that $a^2=0$, $b^2=0$, and $\langle a\cup b, [M]\rangle\neq 0$. 
Then after localization away from any prime $p$ satisfying $2p<{\rm max}(m, n-m)+4$ or $p~|~\langle a\cup b, [M]\rangle$, the following hold:
\begin{itemize}
\item there is a homotopy cofibration
\[
S^{n-1} \larrow S^m\vee S^{n-m}\vee C\stackrel{}{\larrow} M
\]
for a simply connected space $C$ of dimension less than $n-1$;
\item
there is a homotopy fibration
\[
C\rtimes \Omega (S^m\times S^{n-m})\stackrel{}{\larrow} M\stackrel{}{\larrow} S^m\times S^{n-m};
\]
\item the homotopy fibration splits after looping to give a homotopy equivalence 
   \[
   \Omega M\simeq\Omega (S^m\times S^{n-m})\times\Omega \big(C\rtimes \Omega (S^m\times S^{n-m})\big). 
 \] 
 \end{itemize}
\end{theorem}

More generally, let $M$ be an $n$-dimensional simply connected Poincar\'{e} Duality complex with a homotopy cofibration
\[
S^{n-1}\stackrel{\mathfrak{g}}{\larrow} \overline{M}\stackrel{}{\larrow} M.
\]
For any $k\in \mathbb{Z}$, we can define a space $X_{M,k}$ by the homotopy cofibration
\[
S^{n-1}\stackrel{k\mathfrak{g}}{\larrow} \overline{M}\stackrel{}{\larrow} X_{M,k}.
\]
Theorem \ref{pdecom-thm} can be applied to prove the following loop space decomposition theorem for highly connected $X_{M,k}$.   
 \begin{theorem}\label{Mp-thm}
 Let $M$ be an $(l-1)$-connected Poincar\'{e} Duality complex of dimension $n\leq 3l-2$ with $l\geq 3$. Write 
\[
H_i(M;\mathbb{Z})\cong \mathbb{Z}^{\oplus d_i}\oplus T_i,
\]
where $d_i\geq 0$ and $T_i$ is a finite abelian group. 
 Suppose that there exists $(a, b)\in H^{m}(M;\mathbb{Z})\times  H^{n-m}(M;\mathbb{Z})$ such that $l\leq m\leq n-l$, $a^2=0$, $b^2=0$ and $\langle a\cup b, [M]\rangle=\pm 1$. Then after localization away from any prime $p$ satisfying $2p<{\rm max}(m, n-m)+4$, for any $k\in \mathbb{Z}\backslash \{0\}$ the following hold:
\begin{itemize}
\item there is a homotopy cofibration
\[
S^{n-1} \larrow S^m\vee S^{n-m}\vee C\stackrel{}{\larrow} X_{M,k},
\]
where
\[
C\simeq \Big(\mathop{\bigvee}\limits_{d_{m}-1} S^{m}\Big)\vee \Big(\mathop{\bigvee}\limits_{d_{n-m}-1} S^{n-m}\Big) \vee \Big(\mathop{\bigvee}\limits_{\substack{i\neq m, n-m \\ l \leq i \leq n-l}} \mathop{\bigvee}\limits_{d_i} S^{i} \Big) \vee \Big(\mathop{\bigvee}\limits_{l+1\leq i\leq n-l}  P^i(T_i) \Big); 
\]
\item
there is a homotopy fibration
\[
(P^{n}(k)\vee C)\rtimes \Omega (S^m\times S^{n-m})\stackrel{}{\larrow} X_{M,k}\stackrel{}{\larrow} S^m\times S^{n-m};
\]
\item the homotopy fibration splits after looping to give a homotopy equivalence 
   \[
   \Omega X_{M,k}\simeq\Omega (S^m\times S^{n-m})\times\Omega \big((P^{n}(k)\vee C)\rtimes \Omega (S^m\times S^{n-m})\big). 
 \] 
 \end{itemize}
  \end{theorem}
The hypothesis concerning the existence of a pair $(a,b)$ in Theorem \ref{Mp-thm} is relatively mild. For example, if $n$ is even and $H_{\rm odd}(M;\mathbb{Q})$ is nontrivial, or if $n$ is odd and $H_{\rm odd<\frac{n+1}{2}}(M;\mathbb{Q})$ is nontrivial, then Poincar\'{e} duality ensures the existence of such a pair. Moreover, in the case when $k=1$, Theorem \ref{Mp-thm} may be compared with \cite[Theorem 8.6]{ST25}, where the assumptions differ; notably, our theorem does not require eliminating any specific torsion in the homology of $M$. In addition, we establish Theorem \ref{Wp-thm}, an analogue of Theorem \ref{Mp-thm}, for $X_{M,k}$ after killing the torsion in the homology of $M$. Together, Theorems \ref{Mp-thm} and \ref{Wp-thm} may be regarded as generalizations of the corresponding results of Stanton-Theriault \cite[Theorems 1.4 and 8.6]{ST25} and Basu-Basu \cite{BB19}.

$\, $

This paper is organized as follows.
Section \ref{sec: int} contains the proof of Theorem \ref{Zdecom-thm}.
In Section \ref{sec: loc}, we establish Theorem \ref{pdecom-thm} together with its geometric analogue, Theorem \ref{int-thm}.
Section \ref{sec: ap} is devoted to applications:
Subsection \ref{subsec: hyper} proves Theorem \ref{exp-thm} on local hyperbolicity;
Subsection \ref{subsec: inert1} establishes Theorem \ref{Zinert-thm} on inertness and non-inertness, along with its consequence for loop space homology, Theorem \ref{loopH-thm};
Subsection \ref{subsec: refine1} provides a refinement when $\overline{X}$ is a finite-type wedge of spheres and Moore spaces.
Finally, Section \ref{sec: PD} applies our results to Poincar\'{e} duality complexes and their variants $X_{M,k}$:
Subsection \ref{subsec: 2decomM} presents two decompositions arising from Theorems \ref{Zdecom-thm} and \ref{pdecom-thm};
Subsection \ref{subsec: inert2} proves Theorem \ref{QZinert-thm} on the gaps between rational and local or integral inertness;
Subsection \ref{subsec: refine2} proves Theorem \ref{Mp-thm} for highly connected $X_{M,k}$.

$\, $

\noindent{\bf Acknowledgements.} 
The author was supported in part by the National Natural Science Foundation of China (Grant nos. 12331003 and 12288201), the National Key R\&D Program of China (No. 2021YFA1002300) and the Youth Innovation Promotion Association of Chinese Academy Sciences.

The author would like to thank the anonymous referee for a careful reading of the manuscript and for many helpful comments.

\section{An integral case}
\label{sec: int}

Let $X$ be an $n$-dimensional simply connected capped complex with $n\geq 4$. 
Suppose that the $(n-1)$-skeleton $\overline{X}$ of $X$ satisfies  
\[
\overline{X}\simeq S^{m}\vee S^{n-m}\vee C
\]
for a space $C$ with $2\leq m\leq n-2$. Then $C$ is simply connected and of dimension less that $n$, and there is a homotopy cofibration
\begin{equation}\label{Xeq}
S^{n-1}\stackrel{g}{\larrow} S^{m}\vee S^{n-m}\vee C\stackrel{i}{\larrow} X
\end{equation}
where $g$ is the attaching map for the top $n$-cell of $X$ and $i$ is the inclusion of the $(n-1)$-skeleton. Since $S^m \vee S^{n-m}$ is a retract of $S^{m} \vee S^{n-m} \vee C$, the homotopy groups of $S^m \vee S^{n-m}$ split off from those of $S^{m} \vee S^{n-m} \vee C$. Accordingly, one may write $g\simeq g_s+g_r$, where $g_s$ is the component of $g$ on $S^m \vee S^{n-m}$ and $g_r$ collects all remaining components.

Suppose that the attaching map $g$ has the form 
\begin{equation}\label{geq}
g\simeq k\cdot \mathfrak{g} \simeq k\cdot [i_1, i_2]+k\cdot \omega,
\end{equation}
for some $k\in \mathbb{Z}$ and some map $\mathfrak{g}$, whose component on $S^m\vee S^{n-m}$ is the Whitehead product $[i_1,i_2]$ of the canonical inclusions $S^m\stackrel{i_1}{\larrow} S^m\vee S^{n-m}$ and $S^{n-m}\stackrel{i_2}{\larrow} S^m\vee S^{n-m}$, while its remaining components are collected into the map $\omega$.

In this section, we show a loop space decomposition of the capped complex $X$, thereby proving Theorem \ref{Zdecom-thm}.

\subsection{A homotopy cofibration}
\label{subsec: cof}
Our goal in this subsection is to prove a homotopy cofibration in Lemma \ref{cofib-lemma}, which will be used to establish a loop space decomposition of $X$.

Define the space $Q$ and the maps $r'$ and $\iota$ by the following homotopy cofibration diagram
\begin{equation}\label{CXQdiag}
\diagram 
                                                           & S^{n-1}   \rdouble \dto^{g}   & S^{n-1} \dto^{k[i_1, i_2]}\\
 C\rto^<<<<{i_3} \ddouble          & S^{m}\vee S^{n-m}\vee C \rto^<<<<{q_{12}} \dto^{i}  & S^m\vee S^{n-m}         \dto^{\iota}\\
C \rto^<<<<<<{c}              & X               \rto^{r'}              & Q,      
\enddiagram
\end{equation}
where the map $i_3$ is the inclusion into the third wedge summand with $c=i\circ i_3$, the map $q_{12}$ is the pinch map onto the first two wedge summands, the upper right square homotopy commutes by the assumption on $g$, and the middle column is the homotopy cofibration \eqref{Xeq}. In particular, the space $Q$ has the $CW$-structure
$
Q\simeq (S^m\vee S^{n-m})\cup e^n
$, 
and the map $r'$ induces the identity map in the top-dimensional homology.

In general, for a map $f\colon Y \to Z$ between two $n$-dimensional capped complexes $Y$ and $Z$, we define the {\it mapping degree} of $f$ to be the integer $k$ such that $f_\ast([Y])=k [Z]$. This notion generalizes the classical notion of mapping degree between oriented manifolds and is well defined up to sign.

\begin{lemma}\label{degr'lemma}
The mapping degree of the map $r'$ is
\[\hspace{6.7cm}
{\rm deg}(r')=1.
\hspace{6.7cm}\Box\]
\end{lemma}

Recall the Moore space $P^n(k)$ is defined by the homotopy cofibration
\[
S^{n-1}\stackrel{k}{\larrow}S^{n-1} \stackrel{\jmath}{\larrow}P^{n}(k),
\]
where $k$ is the degree $k$ map, and $\jmath$ is the inclusion of the bottom cell. 
In the following homotopy commutative diagram
\[
\diagram 
  S^{n-1}   \rdouble \dto^{k}      & S^{n-1}   \rdouble \dto^{g}   & S^{n-1} \dto^{k[i_1, i_2]}\\
  S^{n-1}  \rto^<<<<{\mathfrak{g}} \dto^{\jmath}       & S^{m}\vee S^{n-m}\vee C \rto^<<<<{q_{12}} \dto^{i}  & S^m\vee S^{n-m}         \dto^{\iota}\\
P^n(k) \rto^<<<<<<{\lambda}              & X               \rto^{r'}              & Q,      
\enddiagram
\]
the columns are homotopy cofibrations, the subdiagram formed by the last two columns is a copy of that in Diagram \eqref{CXQdiag}, and the subdiagram formed by the first two columns is obtained from \eqref{geq} with the induced map $\lambda$. Denote by $\gamma$ the composite along the bottom row
\[
\gamma:  P^{n}(k)\stackrel{\lambda}{\larrow} X\stackrel{r'}{\larrow} Q.
\]
Then the outer diagram of the preceding diagram is the upper part of the following homotopy cofibration diagram
\[\label{Qdiag}
\diagram
S^{n-1} \ddouble \rto^{k}  & S^{n-1} \dto^{[i_1, i_2]} \rto^{\jmath}   & P^{n}(k) \dto^{\gamma} \\
S^{n-1}            \rto^<<<<{k[i_1,i_2]}  & S^m\vee S^{n-m} \dto^{j}  \rto^<<<<<<{\iota}  & Q\dto^{\varphi}  \\
& S^m\times S^{n-m}\rdouble & S^m\times S^{n-m},                       
\enddiagram
\]
where $j$ is the inclusion of the $(n-1)$-skeleton of $S^m\times S^{n-m}$ and $\varphi$ is the induced map. 
\begin{lemma}\label{Qlemma}
There is a homotopy cofibration
\[
P^{n}(k)\stackrel{\gamma}{\larrow} Q\stackrel{\varphi}{\larrow}  S^m\times S^{n-m}
\]
such that the maps $\gamma$ and $\varphi$ restrict to the Whitehead product $[i_1, i_2]$ and the identity map on the $(n-1)$-skeletons, respectively. Moreover, the mapping degree of the map $\varphi$ is
\[
{\rm deg}(\varphi)=k.
\]
\end{lemma}
\begin{proof}
The first statement follows from the preceding homotopy cofibration diagram, and the second statement on the mapping degree follows from the long exact sequence of homology associated to the resulting homotopy cofibration. 
\end{proof}

Denote by $r$ and $\kappa$ the composites
\[
r: X\stackrel{r'}{\larrow} Q\stackrel{\varphi}{\larrow}  S^m\times S^{n-m} \ \  \ {\rm and} \ \ \ 
\kappa: P^{n}(k)\vee C\stackrel{\lambda \vee c}{\llarrow} X\vee X \stackrel{\nabla}{\larrow} X,
\]
respectively, where $\nabla$ is the folding map.

\begin{lemma}\label{degrlemma}
The mapping degree of the map $r$ is 
\[
{\rm deg}(r)=k.
\]
\end{lemma}
\begin{proof}
By Lemmas \ref{degr'lemma} and \ref{Qlemma}, we have 
${\rm deg}(r)={\rm deg}(\varphi)\cdot {\rm deg}(r')=k\cdot 1=k$.
\end{proof}

\begin{lemma}\label{cofib-lemma}
The sequence 
\[
P^{n}(k)\vee C\stackrel{\kappa}{\larrow} X\stackrel{r}{\larrow} S^m\times S^{n-m}
\]
is a homotopy cofibration.
\end{lemma}
\begin{proof}
Define the space $K$ and the map $r''$ by the homotopy cofibration
\[
P^{n}(k)\vee C\stackrel{\kappa}{\larrow} X\stackrel{r''}{\larrow} K.
\]
Consider the long exact sequence of homology associated to the homotopy cofibration
\[
\cdots \rightarrow H_{i}(P^{n}(k)\vee C;\mathbb{Z})\stackrel{\kappa_\ast}{\larrow} H_{i}(X;\mathbb{Z})\stackrel{r''_\ast}{\larrow} H_{i}(K;\mathbb{Z})\stackrel{\delta}{\larrow} H_{i-1}(P^{n}(k)\vee C;\mathbb{Z})\stackrel{\kappa_\ast}{\larrow} H_{i-1}(X;\mathbb{Z}) \rightarrow \cdots .
\]
When $i\geq n+1$, since the dimension of $C$ is less than $n$, it is clear that $H_{\geq n+1}(K;\mathbb{Z})\cong H_{\geq n+1}(X;\mathbb{Z})=0$. 
When $i\leq n-2$, by Diagram \eqref{CXQdiag} and the Five Lemma, the long exact sequence splits into short exact sequences
\[
\spreaddiagramcolumns{-0.7pc}\spreaddiagramrows{0pc} 
\diagram
0 \rto &  H_{i\leq n-2}(P^{n}(k)\vee C;\mathbb{Z}) \rto^<<<<<<{\kappa_\ast}   &  H_{i\leq n-2}(X;\mathbb{Z}) \rto^{r''_\ast}  &  H_{i\leq n-2}(K;\mathbb{Z})\rto^{}  & 0\\
0 \rto & H_{i\leq n-2}(C;\mathbb{Z}) \rto^<<<<<{i_{3\ast}}  \uto^{\cong}_{i_{2\ast}} &   H_{i\leq n-2}(S^m\vee S^{n-m}\vee C;\mathbb{Z})   \rto^<<<{q_{12\ast}} \uto^{\cong}_{i_{\ast}}  & H_{i\leq n-2}(S^m\vee S^{n-m};\mathbb{Z})\rto \uto^{\cong}_{}  & 0,
\enddiagram
\] 
inducing an isomorphism $H_{i\leq n-2}(K;\mathbb{Z}) \cong H_{i\leq n-2}(S^m\vee S^{n-m};\mathbb{Z})$. 
When $i=n$, by \eqref{Xeq} the exact sequence becomes
\[
0\stackrel{}{\larrow} \mathbb{Z}\stackrel{r''_\ast}{\larrow} H_{n}(K;\mathbb{Z})\stackrel{\delta}{\larrow} \mathbb{Z}/k\oplus H_{n-1}(C;\mathbb{Z})\stackrel{\kappa_\ast}{\larrow} H_{n-1}(X;\mathbb{Z})\cong  H_{n-1}(C;\mathbb{Z}),
\] 
where $\kappa_\ast$ restricts to the identity map on $H_{n-1}(C;\mathbb{Z})$. It follows that $H_n(K;\mathbb{Z})\cong \mathbb{Z}$, ${\rm deg}(r'')=k$ and then $H_{n-1}(K;\mathbb{Z})\stackrel{\delta}{\larrow} H_{n-2}(P^{n}(k)\vee C;\mathbb{Z})$ is injective by exactness. However, the connecting homomorphism $\delta$ is trivial by the injectivity of $\kappa_\ast$ on $H_{n-2}$, and so $H_{n-1}(K;\mathbb{Z})=0$. To summarize, we have showed that $H_\ast(K;\mathbb{Z})\cong H_\ast(S^m\times S^{n-m};\mathbb{Z})$, ${\rm deg}(r'')=k$ and $r''_\ast$ restricts to the projection $q_{12\ast}$ on the $(n-1)$-skeletons.

Next, we show that the composite $r\circ \kappa$ is null homotopic. Indeed, the composite $P^n(k)\stackrel{i_1}{\larrow} P^{n}(k)\vee C\stackrel{\kappa}{\larrow} X\stackrel{r}{\larrow}S^m\times S^{n-m}$ is homotopic to the composite 
$
r\circ \lambda\simeq \varphi\circ r' \circ \lambda \simeq \varphi \circ \gamma 
$, 
which is null homotopic by Lemma \ref{Qlemma}. Also, the composite $C\stackrel{i_2}{\larrow} P^{n}(k)\vee C\stackrel{\kappa}{\larrow} X\stackrel{r}{\larrow}S^m\times S^{n-m}$ is homotopic to $r\circ c\simeq \varphi\circ r'\circ c$, which is null homotopic as $C\stackrel{c}{\larrow} X\stackrel{r'}{\larrow} Q$ is a homotopy cofibration by \eqref{CXQdiag}. It follows that the restriction of $r\circ \kappa$ on each wedge summand is null homotopic, and then so is $r\circ \kappa$.

Since $P^{n}(k)\vee C\stackrel{\kappa}{\larrow} X\stackrel{r''}{\larrow} K$ is a homotopy cofibration and $r\circ \kappa$ is null homotopic, the map $r$ can be extended through $r''$ to a map $t: K\stackrel{}{\larrow} S^m\times S^{n-m}$:
\[
r: X\stackrel{r''}{\larrow} K\stackrel{t}{\larrow} S^m\times S^{n-m}.
\]
Recall by the previous discussion $r''_\ast$ restricts to the projection $q_{12\ast}$ on the $(n-1)$-skeletons. Since by Diagram \eqref{CXQdiag} and Lemma \ref{Qlemma} $X\stackrel{r'}{\larrow} Q$ and $Q\stackrel{\varphi}{\larrow} S^m\times S^{n-m}$ restricts to the pinch map $q_{12}$ and the identity map on the $(n-1)$-skeletons, respectively, their composite $r=\varphi\circ r'$ restricts to the pinch map $q_{12}$ on the $(n-1)$-skeletons. Hence, the induced homomorphism $t_\ast: H_{\leq n-1}(K;\mathbb{Z})\larrow H_{\leq n-1}(S^m\times S^{n-m};\mathbb{Z})$ is the identity map. 
Further, by Lemma \ref{degrlemma} and the previous discussion, we have $k={\rm deg}(r)={\rm deg}(t)\cdot {\rm deg}(r'')=k{\rm deg}(t)$, implying that ${\rm deg}(t)=1$. Therefore, the map $t$ induces the identity map on homology, and then is a homotopy equivalence by the Whitehead Theorem. This implies that $P^{n}(k)\vee C\stackrel{\kappa}{\larrow} X\stackrel{r}{\larrow} S^m\times S^{n-m}
$ 
is a homotopy cofibration.
\end{proof}

\subsection{An integral loop space decomposition}
\label{subsec: Zdecom}
To prove a loop space decomposition of $X$ from Lemma \ref{cofib-lemma}, we need to show that $\Omega r$ admits a right homotopy inverse. 

Denote by 
\[
s_1: S^{m}\stackrel{i_1}{\larrow} S^{m}\vee S^{n-m}\vee C \stackrel{i}{\larrow} X
 \ \ \ {\rm and} \ \ \ 
s_2: S^{n-m}\stackrel{i_2}{\larrow} S^{m}\vee S^{n-m}\vee C \stackrel{i}{\larrow} X.
\]
the inclusions of the spheres into $X$, where $i_1$ and $i_2$ denote the canonical inclusions of the spherical wedge summands by abuse of notation. 

\begin{lemma}\label{rlemma}
The composites $S^m\stackrel{s_1}{\larrow} X\stackrel{r}{\larrow}  S^m\times S^{n-m}$ and $S^{n-m}\stackrel{s_2}{\larrow} X\stackrel{r}{\larrow}  S^m\times S^{n-m}$ are homotopic to the canonical inclusions $S^{m}\stackrel{j_1}{\larrow}  S^m\times S^{n-m}$ and $S^{n-m}\stackrel{j_2}{\larrow}  S^m\times S^{n-m}$, respectively.
\end{lemma}
\begin{proof}
By Diagram \eqref{CXQdiag} and Lemma \ref{Qlemma}, there is the homotopy commutative diagram 
\[
\diagram 
S^m \rto^<<<<{i_1} \ddouble          & S^{m}\vee S^{n-m}\vee C \rto^<<<<{q_{12}} \dto^{i}  & S^m\vee S^{n-m}         \dto^{\iota} \rdouble  & S^m\vee S^{n-m}  \dto^{j} \\
S^m \rto^<<<<<<{s_1}              & X               \rto^{r'}              & Q \rto^<<<<<<{\varphi}  & S^m\times S^{n-m}.      
\enddiagram
\]
Then 
$
r\circ s_1\simeq \varphi \circ  r'\circ s_1\simeq  j\circ q_{12}\circ i_1\simeq j_1
$. 
Similarly, $r\circ s_2\simeq j_2$.
\end{proof}

Observe that the inclusion $S^m\vee S^{n-m}\stackrel{i_1\vee i_2}{\llarrow} S^{m}\vee S^{n-m}\vee C \stackrel{i}{\larrow} X$ is homotopic the composite
\[
s: S^m\vee S^{n-m}\stackrel{s_1\vee s_2}{\llarrow}  X\vee X \stackrel{\nabla}{\larrow} X,
\]
where $\nabla$ denotes the folding map. 

\begin{lemma}\label{rlemma2}
The looped map $\Omega r: \Omega X\stackrel{}{\larrow}\Omega (S^m\times S^{n-m})$ has a right homotopy inverse. 
\end{lemma}
\begin{proof}
By the Bott-Samelson Theorem, the Pontryagin algebras
\[
H_\ast(\Omega (S^{m}\vee S^{n-m});\mathbb{Z})\cong T(u, v) \ \ \ {\rm and} \ \ \
H_\ast(\Omega (S^{m}\times S^{n-m});\mathbb{Z})\cong \mathbb{Z}[u, v], 
\]
where $T(u, v)$ is the free tensor algebra on generators $u$ and $v$ of degrees $(m-1)$ and $(n-m-1)$, respectively. Under these isomorphisms, the looped composite
\[
\Omega(S^m\vee S^{n-m})\stackrel{\Omega s}{\larrow}\Omega X\stackrel{\Omega r}{\larrow}\Omega (S^m\times S^{n-m})
\]
induces an algebra homomorphism
\[ 
\Omega (r\circ s)_\ast: T(u, v)\stackrel{}{\longrightarrow} \mathbb{Z}[u, v].
\]
Note that the composite
\[
S^m\stackrel{i_1}{\larrow} S^m\vee S^{n-m}\stackrel{s}{\larrow} X\stackrel{r}{\larrow} S^m\times S^{n-m}
\]
is homotopic to $r\circ s_1$, and hence is homotopic to $j_1$ by Lemma \ref{rlemma}. It follows that $\Omega (r\circ s)_\ast (u)=u$. Similarly, $\Omega (r\circ s)_\ast(v)=v$. Hence, the algebra homomorphism $\Omega (r\circ s)_\ast$ is the abelianization morphism.

Further, by the Hilton-Milnor Theorem, the looped inclusion $\Omega (S^m\vee S^{n-m})\stackrel{\Omega j}{\larrow}\Omega (S^m\times S^{n-m})$ has a right homotopy inverse 
\[
\phi: \Omega (S^m\times S^{n-m})\stackrel{}{\larrow} \Omega(S^m\vee S^{n-m}).
\]
It follows that the algebra homomorphism $(\Omega j)_\ast$ is also the abelianization morphism, and the induced homomorphism $\phi_\ast: H_\ast(\Omega (S^m\times S^{n-m});\mathbb{Z}) \stackrel{}{\larrow} H_\ast(\Omega(S^m\vee S^{n-m});\mathbb{Z})$ on homology is a group monomorphism of $\mathbb{Z}[u, v]$ into $T(u, v)$. Therefore, the composite
\[
\Omega (S^m\times S^{n-m})\stackrel{\phi}{\larrow} \Omega(S^m\vee S^{n-m})\stackrel{\Omega s}{\larrow}\Omega X\stackrel{\Omega r}{\larrow}\Omega (S^m\times S^{n-m})
\]
induces the identity map on homology and hence is a homotopy equivalence by the Whitehead Theorem. 
\end{proof}

To prove Theorem \ref{Zdecom-thm}, we need a general decomposition theorem from~\cite{BT22}. 
Recall the \emph{right half-smash} of two pointed spaces $A$ and $B$ is the quotient space 
\[A\rtimes B=(A\times B)/\sim\] 
where $(\ast,b)\sim(\ast,\ast)$. It is well known that if $A$ is a co-$H$-space then there is a 
homotopy equivalence $A\rtimes B\simeq A\vee (A\wedge B)$.  

\begin{theorem} 
   \label{BT2} 
   Let 
   \(\nameddright{A}{f}{Y}{h}{Z}\) 
   be a homotopy cofibration. Suppose that $\Omega h$ has a right homotopy inverse. Then there is a homotopy fibration 
   \[\llnameddright{A\rtimes\Omega Z}{} 
           {Y}{h}{Z},\] 
           which splits after looping to give a homotopy equivalence 
   \[\hspace{5.75cm}\Omega Y\simeq\Omega Z\times\Omega(A\rtimes\Omega Z). 
         \hspace{5.75cm}\Box\] 
\end{theorem} 

We can now prove Theorem \ref{Zdecom-thm}.

\begin{proof}[Proof of Theorem \ref{Zdecom-thm}]
By Lemma \ref{cofib-lemma}, there is a homotopy cofibration 
\[
P^{n}(k)\vee C\stackrel{\kappa}{\larrow} X\stackrel{r}{\larrow} S^m\times S^{n-m}.
\]
By Lemma \ref{rlemma2}, the map $r$ has a right homotopy inverse after looping. Then Theorem \ref{BT2} can apply and the theorem follows.
\end{proof}

\begin{remark}\label{k=0remark}
In Theorem \ref{Zdecom-thm}, there is a degenerate case when $k=0$. In this case, the attaching map for the top cell of $X$ is null homotopic, and there is a homotopy equivalence
\[
X\simeq S^{m}\vee S^{n-m}\vee S^{n}\vee C.
\]
By Theorem \ref{Zdecom-thm}, there is a homotopy equivalence
\[
   \Omega X\simeq\Omega (S^m\times S^{n-m})\times\Omega \big((S^{n-1}\vee S^n\vee C)\rtimes \Omega (S^m\times S^{n-m})\big), 
\] 
and therefore
\[
\Omega (S^{m}\vee S^{n-m}\vee S^{n}\vee C) \simeq  \Omega (S^m\times S^{n-m})\times\Omega \big((S^{n-1}\vee S^n\vee C)\rtimes \Omega (S^m\times S^{n-m})\big).
\]
Although the homotopy type of $\Omega(S^{m}\vee S^{n-m}\vee S^{n}\vee C)$ is fixed, the particular formulation of the loop space decomposition differs from those obtained via the classical Ganea Theorem or Theorem~\ref{BT2}.
\end{remark}

\begin{remark}\label{k=1remark}
In Theorem \ref{Zdecom-thm}, if $k=1$ then
  \[
   \Omega X\simeq\Omega (S^m\times S^{n-m})\times\Omega \big(C\rtimes \Omega (S^m\times S^{n-m})\big). 
 \] 
This loop space decomposition agrees with that obtained for the Beben-Theriault complexes in \cite{BT14}, and thus provides a partial generalization of their result.
\end{remark}

\begin{remark}\label{natremark}
Theorem \ref{Zdecom-thm} satisfies a naturality property.
Let $X$ and $Y$ be two $n$-dimensional simply connected capped complexes with a homotopy cofibration diagram
\begin{equation}\label{nat-diag}
\diagram
S^{n-1} \dto^{t} \rto^<<<{k\mathfrak{g}}   &  S^{m}\vee S^{n-m}\vee C  \dto^{1\vee 1\vee \mathfrak{f}}  \rto^<<<<{i_X}  &  X\dto^{f} \\
S^{n-1}            \rto^<<<{\ell\mathfrak{h}}   &  S^{m}\vee S^{n-m}\vee D   \rto^<<<<{i_Y}  & Y,
\enddiagram
\end{equation}
for some spaces $C$ and $D$, $k, \ell, t\in \mathbb{Z}$ with $k=t \ell$ and $2\leq m\leq n-2$. Suppose that the components of $\mathfrak{g}$ and $\mathfrak{h}$ on $S^m\vee S^{n-m}$ are the Whitehead product of the inclusions $S^m\stackrel{i_1}{\larrow} S^m\vee S^{n-m}$ and $S^{n-m}\stackrel{i_2}{\larrow} S^m\vee S^{n-m}$.

We may go through the proof of Theorem \ref{Zdecom-thm} for both $X$ and $Y$ simultaneously, taking Diagram \eqref{nat-diag} into account. 
Since $k=t \ell$, there is the homotopy cofibration diagram
\[
\diagram
S^{n-1} \dto^{t} \rto^<<<{k[i_1, i_2]}   &  S^{m}\vee S^{n-m}  \ddouble  \rto^{}  &  Q_{k}\dto^{\varphi_{k\ell}} \\
S^{n-1}            \rto^<<<{\ell[i_1, i_2]}   &  S^{m}\vee S^{n-m}  \rto^{}  & Q_{\ell},
\enddiagram
\]
where the induced map $\varphi_{k\ell}$ is of degree $t$. By Diagram \eqref{CXQdiag}, there is a homotopy cofibration diagram
\[
\diagram
C \rto^{c_X} \dto^{\mathfrak{f}}  &  X \dto^{f} \rto^{r'_X}  & Q_{k}\dto^{\varphi_{k\ell}} \\
D                 \rto^{c_Y}               &             Y        \rto^{r'_Y}  & Q_{\ell}.
\enddiagram
\]
Also, from Diagram \eqref{nat-diag} the maps $i_X\circ \mathfrak{g}$ and $i_Y\circ \mathfrak{h}$ extend to the maps $\lambda_X$ and $\lambda_Y$ such that the diagram
\[
\diagram
P^n(k) \rto^<<<<{\lambda_X} \dto^{\jmath_{t}} & X \dto^{f}\\
P^n(\ell) \rto^<<<<{\lambda_Y}  &Y
\enddiagram
\]
homotopy commutes, where $\jmath_{t}$ is the identity map on the bottom cell and is of degree $t$ on the top cell. 
Denote by $\gamma_X$ and $\gamma_Y$ the composites
\[
\gamma_X:  P^{n}(k)\stackrel{\lambda_X}{\larrow} X\stackrel{r'_X}{\larrow} Q_{k} \ \ \ {\rm and} \ \ \
\gamma_Y:  P^{n}(\ell)\stackrel{\lambda_Y}{\larrow} Y\stackrel{r'_Y}{\larrow} Q_{\ell}.
\]
Then by Lemma \ref{Qlemma}, there is a homotopy cofibration diagram 
\[
\diagram
P^n(k) \rto^<<<<{\gamma_X} \dto^{\jmath_{t}} & Q_{k} \dto^{\varphi_{k\ell}} \rto^<<<<{\varphi_{k1}}  & S^m\times S^{n-m} \ddouble \\
P^n(\ell) \rto^<<<<{\gamma_Y}  &Q_{\ell} \rto^<<<<{\varphi_{\ell1}}  & S^m\times S^{n-m}  . 
\enddiagram
\]
Denote by $r_X$, $\kappa_X$, $r_Y$ and $\kappa_Y$ the composites
\[
\begin{split}
r_X: X\stackrel{r'_X}{\larrow} Q_k\stackrel{\varphi_{k1}}{\larrow}  S^m\times S^{n-m} \ \  \ &{\rm and} \ \ \ 
\kappa_X: P^{n}(k)\vee C\stackrel{\lambda_X \vee c_X}{\llarrow} X\vee X \stackrel{\nabla}{\larrow} X, \\
r_Y: Y\stackrel{r'_Y}{\larrow} Q_\ell\stackrel{\varphi_{\ell1}}{\larrow}  S^m\times S^{n-m} \ \  \ &{\rm and} \ \ \ 
\kappa_Y: P^{n}(\ell)\vee D\stackrel{\lambda_Y \vee c_Y}{\llarrow} Y\vee Y \stackrel{\nabla}{\larrow} Y, 
\end{split}
\]
respectively.
By Lemma \ref{cofib-lemma}, there is a homotopy cofibration diagram 
\begin{equation}
\label{nat-cofib-diag}
\diagram
P^n(k)\vee C \rto^<<<<{\kappa_X} \dto^{\jmath_{t}\vee \mathfrak{f}} & X \dto^{f} \rto^<<<<{r_X}  & S^m\times S^{n-m} \ddouble \\
P^n(\ell)\vee D \rto^<<<<{\kappa_Y}  &Y \rto^<<<<{r_Y}  & S^m\times S^{n-m}. 
\enddiagram
\end{equation}

We claim that the maps $\Omega r_X$ and $\Omega r_Y$ have compatible right homotopy inverses. Indeed, in the proof of Lemma \ref{rlemma2}, we have showed that the composites
\[
\begin{split}
&\Omega (S^m\times S^{n-m})\stackrel{\phi}{\larrow} \Omega(S^m\vee S^{n-m})\stackrel{\Omega s_X}{\larrow}\Omega X \stackrel{\Omega r_X}{\larrow}\Omega (S^m\times S^{n-m}),  \\ 
&\Omega (S^m\times S^{n-m})\stackrel{\phi}{\larrow} \Omega(S^m\vee S^{n-m})\stackrel{\Omega s_Y}{\larrow}\Omega Y \stackrel{\Omega r_Y}{\larrow}\Omega (S^m\times S^{n-m}) 
\end{split}
\]
induce the identity map on homology. Recall the inclusions $s_X$ and $s_Y$ factors as 
\[
s_X: S^m\vee S^{n-m}\stackrel{i_{12}}{\larrow} S^m\vee S^{n-m} \vee C\stackrel{i_X}{\larrow} X  \ \ {\rm and}   \ \
s_Y: S^m\vee S^{n-m}\stackrel{i_{12}}{\larrow} S^m\vee S^{n-m} \vee D \stackrel{i_Y}{\larrow} Y. 
\]
Then in the homotopy commutative diagram
\[
\diagram
\Omega(S^m\times S^{n-m}) \ddouble \rto^{\phi}   & \Omega (S^m\vee S^{n-m}) \ddouble \rto^<<<<{\Omega i_{12}}  & \Omega (S^m\vee S^{n-m} \vee C) \dto^{\Omega (1\vee 1\vee \mathfrak{f})} \rto^<<<<{\Omega i_X}  & \Omega X\dto^{\Omega f} \rto^<<<<{\Omega r_X}  & \Omega(S^m\times S^{n-m}) \ddouble\\
\Omega(S^m\times S^{n-m}) \rto^{\phi}   & \Omega (S^m\vee S^{n-m})  \rto^<<<<{\Omega i_{12}}  & \Omega (S^m\vee S^{n-m} \vee D)  \rto^<<<<{\Omega i_Y}  & \Omega Y \rto^<<<<{\Omega r_Y}  & \Omega(S^m\times S^{n-m}),
\enddiagram
\]
the two row composites induce the identity map on homology, and hence are compatible homotopy equivalences by the Whitehead Theorem. It follows that the maps $\Omega r_X$ and $\Omega r_Y$ have compatible right homotopy inverses. 

Therefore, for the homotopy cofibration diagram \eqref{nat-cofib-diag}, we can apply Theorem \ref{BT2} and its naturality property (see \cite[Remark 2.2]{HT22}, \cite[Remark 2.7]{The24a}, or \cite[Proposition 3.3]{Hua24}) to obtain a homotopy fibration diagram 
\[
\diagram
(P^{n}(k)\vee C)\rtimes \Omega (S^m\times S^{n-m}) \rto^{} \dto^{(\jmath_{t}\vee \mathfrak{f})\rtimes 1} &X \dto^{f} \rto^<<<<{r_X}  & S^m\times S^{n-m} \ddouble \\
(P^{n}(k)\vee D)\rtimes \Omega (S^m\times S^{n-m}) \rto^{}  &Y \rto^<<<<{r_Y}  & S^m\times S^{n-m},  
\enddiagram
\]
in which the rows splits after looping to give compatible homotopy equivalences
   \[
   \diagram
   \Omega (S^m\times S^{n-m})\times\Omega \big((P^{n}(k)\vee C)\rtimes \Omega (S^m\times S^{n-m})\big) \rto^<<<<{\simeq}  \dto^{1\times \Omega ((\jmath_{t}\vee \mathfrak{f})\rtimes 1)} & \Omega X \dto^{\Omega f}\\
      \Omega (S^m\times S^{n-m})\times\Omega \big((P^{n}(\ell)\vee D)\rtimes \Omega (S^m\times S^{n-m})\big) \rto^<<<<{\simeq}   &  \Omega Y.
   \enddiagram
 \] 
 
In conclusion, Theorem \ref{Zdecom-thm} is natural for maps of homotopy cofibrations as in Diagram \eqref{nat-diag}.
\end{remark}


\section{A local case}
\label{sec: loc}

In this section, we establish a local analogue of Theorem \ref{Zdecom-thm}, namely Theorem \ref{pdecom-thm}, and apply it to smooth manifolds with two essentially intersecting embedded spheres, thereby proving Theorem \ref{int-thm}.

\subsection{Two lifting lemmas}
\label{subsec: lift}
We start with two lifting lemmas in local category. 
Denote by $S^l\stackrel{i}{\larrow} K(\mathbb{Z}, l)$ the inclusion of the bottom cell of $K(\mathbb{Z}, l)$.

\begin{lemma} 
   \label{oddlift-lemma} 
   Let $X$ be a path-connected $CW$-complex of dimension $n$. Suppose that $a\in H^{2l+1}(X;\mathbb{Z})$ is represented by a map $f: X\stackrel{}{\larrow} K(\mathbb{Z}, 2l+1)$. Then after localization away from any prime $p$ satisfying $2l+2p-2<n$, there is a lift
   \[
   \diagram
   & S^{2l+1} \dto^{i} \\
   X\urto^{\widetilde{f}} \rto^<<<<<{f}  & K(\mathbb{Z}, 2l+1)
   \enddiagram
   \]
  for some map $\widetilde{f}$. 
  
Furthermore, after localization away from any prime $p$ satisfying $2l+2p-3<n$, the lift of $f$ is unique up to homotopy. In particular, if $\widetilde{f}^\ast: H^{2l+1}(S^{2l+1};\mathbb{Z})\stackrel{}{\larrow}  H^{2l+1}(X;\mathbb{Z})$ is trivial, then $\widetilde{f}$ is null homotopic.
\end{lemma} 
\begin{proof}
The first statement follows from \cite[Lemma 2.2]{Hua25} in which the finite condition is clearly not necessary, and the second statement follows by applying the first statement to $X\times I$.
\end{proof}

\begin{lemma} 
   \label{evenlift-lemma} 
   Let $X$ be a path-connected $CW$-complex of dimension $n$. Suppose that $a\in H^{2l}(X;\mathbb{Z})$ is represented by a map $f: X\stackrel{}{\larrow} K(\mathbb{Z}, 2l)$ and $a^2=0$. Then after localization away from any prime $p$ satisfying $2l+2p-4<n$, there is a lift
   \[
   \diagram
   & S^{2l} \dto^{i} \\
   X\urto^{\widetilde{f}} \rto^<<<<<{f}  & K(\mathbb{Z}, 2l)
   \enddiagram
   \]
  for some map $\widetilde{f}$. 
  
Furthermore, after localization away from any prime $p$ satisfying $2l+2p-5<n$, the lift of $f$ is unique up to homotopy. In particular, if $\widetilde{f}^\ast: H^{2l}(S^{2l};\mathbb{Z})\stackrel{}{\larrow}  H^{2l}(X;\mathbb{Z})$ is trivial, then $\widetilde{f}$ is null homotopic.
\end{lemma} 
\begin{proof}
The first statement follows from \cite[Lemma 7.4]{The24b}. Note that the localization here is taken away from any prime $p$ with $2l+2p-4<n$, in contrast to the condition $2l+2p-4\leq n$ used there. This refinement arises because in the proof it suffices to use the connectivity of the homotopy fibre of $\Omega S^{2l+1}\larrow K(\mathbb{Z}, 2l)$, rather than requiring the map to be a homotopy equivalence. The second statement follows by applying the first statement to $X\times I$.
\end{proof}

\subsection{A local loop space decomposition}
\label{subsec: pdecom}

Let $X$ be an $n$-dimensional simply connected capped complex with a fundamental class $[X]$. 
Suppose that $(a, b)\in H^{m}(X;\mathbb{Z})\times  H^{n-m}(X;\mathbb{Z})$ with $2\leq m\leq n-2$ is a {\it spherical pair} of $X$, that is, $a^2=0$ and $b^2=0\in H^\ast(X;\mathbb{Z})$, and there exist maps
\[
s_1: S^m\larrow X \ \ \ {\rm and} \ \ \  s_2: S^{n-m}\larrow X
\]
such that 
$s_1^\ast(a)$ and $s_2^\ast(b)$ are generators of $H^{m}(S^m;\mathbb{Z})$ and $H^{n-m}(S^{n-m};\mathbb{Z})$, respectively.

\begin{lemma}\label{plemma}
After localization away from any prime $p$ satisfying $2p<{\rm max}(m, n-m)+4$, the maps $s_1$ and $s_2$ have left homotopy inverses  
\[
r_1: X\stackrel{}{\larrow} S^{m} \ \ \ {\rm and} \ \ \  r_2: X\stackrel{}{\larrow} S^{n-m}, 
\]
respectively. Suppose further that the $(n-1)$-skeleton $\overline{X}$ of $X$ is a co-$H$-space. Then 
\[
\overline{X}\simeq S^{m}\vee S^{n-m}\vee C
\]
for a simply connected space $C$ of dimension less than $n$, such that the cohomology classes $a\in H^{m}(X;\mathbb{Z})$ and $b\in H^{n-m}(X;\mathbb{Z})$ correspond to the spherical wedge summands $S^m$ and $S^{n-m}$, respectively.
\end{lemma}
\begin{proof}
Localize away from any prime $p$ satisfying $2p<{\rm max}(m, n-m)+4$. 

(1). Denote by $X\stackrel{a}{\larrow}K(m, \mathbb{Z})$ and $X\stackrel{b}{\larrow}K(n-m, \mathbb{Z})$ the two maps representing the spherical elements $a$ and $b$, respectively. 
Then the composites 
\[
S^m\stackrel{s_1}{\larrow} X\stackrel{a}{\larrow}K(m, \mathbb{Z}) \ \ \ {\rm and} \ \ \  
S^{n-m}\stackrel{s_2}{\larrow} X\stackrel{b}{\larrow}K(n-m, \mathbb{Z})
\]
represent the generators of $H^{m}(S^m;\mathbb{Z})$ and $H^{n-m}(S^{n-m};\mathbb{Z})$, respectively, and hence are homotopic to the bottom cell inclusions. 
By Lemmas \ref{oddlift-lemma} and \ref{evenlift-lemma}, the maps $a$ and $b$ can be lifted to some maps 
\[
r_1: X\stackrel{}{\larrow} S^{m} \ \ \ {\rm and} \ \ \  r_2: X\stackrel{}{\larrow} S^{n-m}, 
\]
respectively. Then $r_1\circ s_1$ and $r_2\circ s_2$ lifts the maps $a\circ s_1$ and $b\circ s_2$, respectively. Since the latter maps are homotopic to the bottom cell inclusions, it follows that $r_1\circ s_1$ and $r_2\circ s_2$ are degree one self-maps, and therefore are homotopic to the identity maps. 

(2). Observe that by degree reasons the maps $s_1$ and $s_2$ factor through the $(n-1)$-skeleton $\overline{X}$ of $X$ to define maps $\overline{s}_1: S^{m}\stackrel{}{\larrow}  \overline{X}$ and $\overline{s}_2: S^{n-m}\stackrel{}{\larrow}  \overline{X}$. By (1) they have left homotopy inverses 
\[
\overline{r}_1 : \overline{X} \stackrel{}{\larrow}X \stackrel{r_1}{\larrow} S^{m} \ \ \ {\rm and} \ \ \ 
\overline{r}_2: \overline{X} \stackrel{}{\larrow} X \stackrel{r_2}{\larrow} S^{n-m},
\]
respectively. 
Define the space $C$ and the map $\overline{q}$ by the homotopy cofibration
\[
S^{m}\vee S^{n-m}\stackrel{\overline{s}}{\larrow}  \overline{X} \stackrel{\overline{q}}{\larrow} C,
\]
where $\overline{s}$ denotes the composite $S^m\vee S^{n-m}\stackrel{\overline{s}_1\vee \overline{s}_2}{\llarrow}  \overline{X} \vee \overline{X}  \stackrel{\nabla}{\larrow} \overline{X}$. Then the composite
\[
\overline{X}\stackrel{\sigma}{\larrow} \overline{X}\vee \overline{X} \vee \overline{X}\stackrel{\overline{r}_1\vee \overline{r}_2\vee \overline{q}}{\lllarrow} S^{m}\vee S^{n-m}\vee C
\]
induces an isomorphism on homology, where $\sigma$ is the iterated comultiplication of the co-$H$-space $\overline{X}$. By the Whitehead Theorem, the composite is a homotopy equivalence. Under this homotopy equivalence, it is clear that the cohomology classes $a\in H^{m}(X;\mathbb{Z})$ and $b\in H^{n-m}(X;\mathbb{Z})$ correspond to the spherical wedge summands $S^m$ and $S^{n-m}$, respectively.
\end{proof}

We can now prove Theorem \ref{pdecom-thm}.

\begin{proof}[Proof of Theorem \ref{pdecom-thm}]
Localize away from any prime $p$ satisfying $2p<{\rm max}(m, n-m)+4$. Since the $(n-1)$-skeleton $\overline{X}$ of $X$ is a co-$H$-space, by Lemma \ref{plemma} $\overline{X}\simeq S^{m}\vee S^{n-m}\vee C$ and there is a homotopy cofibration
\[
S^{n-1} \stackrel{g}{\larrow} S^m\vee S^{n-m}\vee C\stackrel{}{\larrow} X
\]
for some attaching map $g$.

By assumption, the attaching map $g\simeq k\mathfrak{g}$ for some map $S^{n-1} \stackrel{\mathfrak{g}}{\larrow} S^m\vee S^{n-m}\vee C$. By the Hilton-Milnor Theorem, the component of $\mathfrak{g}$ on $S^m\vee S^{n-m}$ 
\[
S^{n-1} \stackrel{\mathfrak{g}}{\larrow} S^m\vee S^{n-m}\vee C\stackrel{q_{12}}{\larrow} S^m\vee S^{n-m}
\]
has the form $\mathfrak{g}_1+\mathfrak{g}_2+\ell [i_1, i_2]$ for some $\ell\in \mathbb{Z}$, and some maps $S^{n-1}\stackrel{\mathfrak{g}_1}{\larrow} S^m$ and $S^{n-1}\stackrel{\mathfrak{g}_2}{\larrow} S^{n-m}$. 
Observe that $\mathfrak{g}_1$ and $\mathfrak{g}_2$ induce trivial homomorphisms on cohomology, and the maps are localized away from any prime $p$ satisfying $2p<{\rm max}(m, n-m)+4$. Then the uniqueness of lifting in Lemmas \ref{oddlift-lemma} and \ref{evenlift-lemma} can apply to show that $\mathfrak{g}_1$ and $\mathfrak{g}_1$ are null homotopic. It follows that 
\[
 g\simeq k\mathfrak{g}\simeq k\ell [i_1, i_2]+k\omega,
\]
where $\omega$ collects all other component of $\mathfrak{g}$. 

In particular, there is the homotopy cofibration diagram 
\[
\diagram
S^{n-1} \rto^<<<<{g} \ddouble   & S^m\vee S^{n-m}\vee C \rto^{} \dto^{q_{12}}  & X\dto^{}\\
S^{n-1} \rto^<<<<{k\ell [i_1, i_2]}  \dto^{k\ell}  & S^m\vee S^{n-m} \rto^{}  \ddouble  & Q' \dto^{}\\
S^{n-1}     \rto^<<<<{ [i_1, i_2]}   & S^m\vee S^{n-m} \rto^{j} & S^m\times S^{n-m} 
\enddiagram
\]
defining the space $Q'$. It implies that 
\[
k\ell=\langle a\cup b, [Q']\rangle=\langle a\cup b, [X]\rangle,
\]
where by abuse of notation the cohomology classes $a$ and $b$ correspond to the spherical wedge summands $S^m$ and $S^{n-m}$, respectively. Since $\langle a\cup b, [X]\rangle\neq 0$ by assumption, we see that $k$, $\ell\neq 0$. 

Localize further away from any prime $p$ satisfying $p~|~\ell$. Denote 
\[
g'=k\mathfrak{g}'=k[i_1, i_2]+\frac{k}{l}\omega.
\]
Then $g' \circ \ell \simeq kl[i_1, i_2]+k\omega\simeq g$, and there is the homotopy cofibration diagram 
\[
\diagram
S^{n-1} \rto^<<<<{g} \dto^{\ell}   & S^m\vee S^{n-m}\vee C \rto^{} \ddouble  & X\dto^{}\\
S^{n-1} \rto^<<<<{g'}                                          & S^m\vee S^{n-m}\vee C  \rto^{}                                  & X',
\enddiagram
\]
defining the space $X'$. Since $\ell$ is a homotopy equivalence, the Five Lemma implies that $X\stackrel{}{\larrow} X'$ induces an isomorphism on homology. Therefore, it is a homotopy equivalence by the Whitehead Theorem. In particular, there is a homotopy cofibration 
\[
S^{n-1} \stackrel{k\mathfrak{g}'}{\llarrow} S^m\vee S^{n-m}\vee C\stackrel{}{\larrow} X
\]
such that the component of $\mathfrak{g}'$ on $S^m\vee S^{n-m}$ is the Whitehead product $[i_1, i_2]$. Then Theorem \ref{Zdecom-thm} can apply and the theorem follows.
\end{proof}

\subsection{A geometric case}
\label{subsec: geo}
Theorem \ref{pdecom-thm} can be applied to a geometric case when two embedded spheres intersect essentially in a manifold. 

\begin{proof}[Proof of Theorem \ref{int-thm}]
It is clear that a simply connected closed manifold is a capped complex. Localize away from any prime $p$ satisfying $2p<{\rm max}(m, n-m)+4$ or $p~|~\langle a\cup b, [M]\rangle$. Then the restrictions of the cohomology classes $a=(s_{1\ast}([S^m]))^\ast$ and $b=(s_{2\ast}([S^{n-m}]))^\ast$ on the embedded spheres $S^m$ and $S^{n-m}$ are the generators of the homology $H^{m}(S^m)$ and $H^{n-m}(S^{n-m})$, respectively. Combining assumption, it follows that $(a,b)$ is a spherical pair of $M$ such that $\langle a\cup b, [M]\rangle \neq 0$. Hence, Theorem \ref{pdecom-thm} with $k=1$ can apply to prove the theorem.
\end{proof}


\section{Applications}
\label{sec: ap}

In this section, we present applications to local hyperbolicity, to inertness and non-inertness, and to a refinement of Theorem \ref{Zdecom-thm} in the case when $\overline{X}$ is a finite-type wedge of spheres and Moore spaces. In particular, we establish Theorems \ref{exp-thm}, \ref{Zinert-thm}, and \ref{loopH-thm}.

\subsection{Local hyperbolicity}
\label{subsec: hyper}
To prove Theorem \ref{exp-thm} on the local hyperbolicity of the capped complex $X$ in Theorem \ref{Zdecom-thm}, we need an auxiliary lemma. 
\begin{lemma}\label{semi-lemma}
Let $A$, $B$ and $Y$ be three path-connected spaces. Then there is a natural homeomorphism
\[
(A\vee B)\rtimes Y\cong  (A\rtimes Y )\vee (B\rtimes Y ).
\]
\end{lemma}
\begin{proof}
 Observe that 
 \[
 (A\vee B)\rtimes Y\cong (A\times Y)\cup_{(\ast\times Y)}(B\times Y).
 \]
 Collapsing $Y$ to a point then gives
 \[
 (A\vee B)\rtimes Y\cong (A\rtimes Y )\cup_{(\ast\rtimes \ast)} (B\rtimes Y )=  (A\rtimes Y )\vee (B\rtimes Y ).
 \]
 This identification is obviously natural.
\end{proof}

The following important result was proved by Huang-Wu \cite{HW20} and Boyde \cite{Boy24}.
\begin{lemma}\label{BHWlemma}
Let $p$ be a prime, $n\geq 3$ and $r\geq 1$. For the Moore space $P^{n}(p^r)$ the following hold: 
\begin{itemize}
\item if $p$ is odd, then $P^{n}(p^r)$ is $\mathbb{Z}/p^s$-hyperbolic for any $s\leq r+1$;
\item if $p=2$ and $r\geq 2$, then $P^{n}(2^r)$ is $\mathbb{Z}/2^s$-hyperbolic for any $s\leq r+1$;
\item if $p=2$ and $r=1$, then $P^{n}(2)$ is $\mathbb{Z}/2$-, $\mathbb{Z}/4$- and $\mathbb{Z}/8$-hyperbolic.
\end{itemize}
\end{lemma}
\begin{proof}
The lemma is a combination of \cite[Theorem 1.6]{HW20} and \cite[Theorem 1.3]{Boy24}.
\end{proof}

We are now in a position to prove Theorem \ref{exp-thm}; a stronger statement concerning homotopy groups will be established later in Lemma \ref{Znoninert-lemma}. 

\begin{proof}[Proof of Theorem \ref{exp-thm}]
Recall for the capped complex $X$ in Theorem \ref{Zdecom-thm} there is a homotopy cofibration
\[
S^{n-1}\stackrel{k\mathfrak{g}}{\larrow} S^{m}\vee S^{n-m}\vee C\stackrel{}{\larrow} X.
\]

If $k=0$, the homotopy cofibration for $X$ splits and there is a homotopy equivalence
\[
X\simeq S^{m}\vee S^{n-m}\vee C \vee S^{n}.
\]
By \cite[Theorem 1]{Boy22}, $S^{m}\vee S^{n-m}$ is $\mathbb{Z}/p^r$-hyperbolic for all primes $p$ and all $r\geq 1$. Then so is $X$. Also, it is well known that $S^{m}\vee S^{n-m}$ is rationally hyperbolic, and so is $X$.

If $k\neq 0$, $\pm 1$, by Theorem \ref{Zdecom-thm} there is a homotopy equivalence
\[
   \Omega X\simeq\Omega (S^m\times S^{n-m})\times\Omega \big((P^{n}(k)\vee C)\rtimes \Omega (S^m\times S^{n-m})\big). 
   \]
   By Lemma \ref{semi-lemma}, $P^{n}(k)\rtimes \Omega (S^m\times S^{n-m})$ retracts off $(P^{n}(k)\vee C)\rtimes \Omega (S^m\times S^{n-m})$, and then $\Omega (P^{n}(k)\rtimes \Omega (S^m\times S^{n-m}))$ retracts off $\Omega X$. However, by the James suspension splitting theorem, 
      \[\begin{split}
   P^{n}(k)\rtimes \Omega (S^m\times S^{n-m})
   &\simeq P^{n}(k)\vee \Big( P^{n}(k) \wedge \big(  \Omega S^m  \vee \Omega S^{n-m}  \vee(  \Omega S^m \wedge \Omega S^{n-m}  ) \big)  \Big)\\
   &\simeq P^{n}(k)\vee \Big( P^{n}(k)\wedge  \mathop{\bigvee}\limits_{j=1}^{\infty}\big(S^{j(m-1)}\vee S^{j(n-m-1)}\big) \Big)  \\ 
   & \ \ \ \ \ \ \ \ \ \  \ \ \vee  \Big( P^{n}(k) \wedge  \mathop{\bigvee}\limits_{i, j=1}^{\infty}\big(S^{i(m-1)+j(n-m-1)}\big) \Big),
   \end{split}
   \]
   and hence is a homotopy equivalent to a finite-type wedge of Moore spaces $P^{n\ast}(k)$. 
Since $P^{n\ast}(k)\simeq \mathop{\bigvee}\limits_{j=1}^\ell P^{n\ast}(p_j^{t_j})$ if $k=\mathop{\prod}\limits_{j=1}^\ell p_j^{t_j}$, the theorem in this case follows from Lemma \ref{BHWlemma}. 
\end{proof}

\subsection{Inertness and non-inertness, I}
\label{subsec: inert1}
In this subsection, we prove Theorems \ref{loopH-thm} and \ref{Zinert-thm}. 
Let $A\stackrel{h}{\longrightarrow} Y\stackrel{\varphi}{\longrightarrow} Z$ be a homotopy cofibration. Recall the map $h$ is called {\it inert} if $\Omega \varphi$ has a right homotopy inverse. 

 We start with the inert case.
 
 \begin{lemma}\label{Zinert-lemma}
 Let $X$ be the capped complex in Theorem \ref{Zdecom-thm}. If $k=\pm 1$, then the attaching map for the top cell of $X$ is inert. 
 \end{lemma}
 \begin{proof}
When $k=\pm 1$, Diagram \eqref{CXQdiag} becomes the homotopy cofibration diagram 
\[
\diagram 
                                                           & S^{n-1}   \rdouble \dto^{\mathfrak{g}}   & S^{n-1} \dto^{\pm [i_1, i_2]}\\
 C\rto^<<<<{i_3} \ddouble          & S^{m}\vee S^{n-m}\vee C \rto^<<<<{q_{12}} \dto^{i}  & S^m\vee S^{n-m}         \dto^{j}\\
C \rto^<<<<<<{c}              & X               \rto^<<<<<<<{r}              & S^m\times S^{n-m}.      
\enddiagram
\]
By the Hilton-Milnor Theorem, the map $j$ has a right homotopy inverse after looping. Also, since $q_{12}$ has a right homotopy inverse, the map $i_3$ is inert. Then by \cite[Theorem 5.4]{Hua24}, the map $i$ has a right homotopy inverse after looping, that is, the attaching map $\mathfrak{g}$ is inert. 
 \end{proof}

The inertness result in Lemma \ref{Zinert-lemma} is useful for computing loop space homology. For this, there is a general result in \cite[Proposition 10.1]{The24a} which is slightly improved in \cite[Section 4]{ST25}.
\begin{proposition}\label{ST-prop}
Let 
\[
\Sigma A \stackrel{f}{\larrow} Y\stackrel{}{\larrow} Z
\]
be a homotopy cofibration such that $f$ is inert and $Y$ is a co-$H$-space. 
Let $A\stackrel{\widetilde{f}}{\larrow} \Omega Y$ be the adjoint of $f$. Then for any commutative ring $R$ with unit such that $H_\ast(Y; R)$ is a free $R$-module, there is an algebra isomorphism
\[
H_\ast(\Omega Z; R)\cong T(\Sigma^{-1} \widetilde{H}_\ast(Y); R)/ (\text{Im}~(\widetilde{f}_\ast)),
\]
where $(\text{Im}~(\widetilde{f}_\ast))$ is the two sided ideal generated by $\text{Im}~(\widetilde{f}_\ast)$. Moreover, if $Y$ is the suspension of a co-$H$-space, then this is an isomorphism of Hopf algebras.  ~$\qqed$
\end{proposition}

We can now prove Theorem \ref{loopH-thm}.

\begin{proof}[Proof of Theorem \ref{loopH-thm}]
Since $k=\pm1$, there is the homotopy cofibration
\[
S^{n-1}\stackrel{\pm \mathfrak{g}}{\larrow} S^{m}\vee S^{n-m}\vee C\stackrel{}{\larrow} X
\]
By Lemma \ref{Zinert-lemma}, the attaching map $\pm \mathfrak{g}$ is inert. Since $C$ is a co-$H$-space with $H_\ast(C; R)$ a free $R$-module by assumption, so is $S^{m}\vee S^{n-m}\vee C$. Further, if $C$ is the suspension of a co-$H$-space, so is $S^{m}\vee S^{n-m}\vee C$. Then Proposition \ref{ST-prop} can apply and the theorem follows.
\end{proof}

We now turn to the non-inert case, beginning with the case of Moore spaces. Indeed, the cell attachment defining a Moore space produces infinitely many new homotopy groups, whose ranks grow exponentially.

 \begin{lemma}\label{cokerlemma}
Let $p$ be a prime, $n\geq 3$ and $r\geq 1$. 
Consider the graded cokernel of the homomorphism 
\[
\jmath_\ast: \pi_\ast(S^{n-1})\stackrel{}{\larrow} \pi_\ast(P^n(p^r))
\]
of homotopy groups induced by the bottom cell inclusion $\jmath$. Then in any of the following cases: 
\begin{itemize}
\item $p$ is odd and $s\leq r+1$;
\item $p=2$, $r\geq 2$ and $s\leq r+1$;
\item $p=2$, $r=1$ and $s=2$, $4$, or $8$,
\end{itemize}
the number of $\mathbb{Z}/p^s$-summands in the graded cokernel grows exponentially. 
 \end{lemma}
 \begin{proof}
To prove the lemma, it is equivalent to consider the induced homomorphism
\[
\widetilde{\jmath}_\ast: \pi_\ast(S^{n-2})\stackrel{}{\larrow} \pi_\ast(\Omega P^n(p^r)),
\]
where the bottom cell inclusion $\widetilde{\jmath}$ is the adjoint of $\jmath$. 

We need four loop space decompositions of Moore spaces. When $p$ is an odd prime and $n=2n'+2$ is even, by \cite[Theorem 1.1]{CMN79}
\[
\Omega P^{2n'+2} (p^r)\simeq S^{2n'+1}\{p^r\}\times \Omega \Big( \mathop{\bigvee}\limits_{j=0}^{\infty} P^{4n'+2jn'+3} (p^r)\Big),
\]
 where $S^{2n'+1}\{p^r\}$ is the homotopy fibre of the degree $p^r$ self-map of $S^{2n'+1}$; when $p$ is an odd prime and $n=2n'+1$ is odd, by \cite{CMN87}
 \[
 \Omega P^{2n'+1} (p^r)\simeq T^{2n'+1}\{p^r\}\times \Omega \Big( \mathop{\bigvee}\limits_{\alpha}^{} P^{n'_\alpha+1} (p^r)\Big)
 \]
for a space $T^{2n'+1}\{p^r\}$, where $\mathop{\bigvee}\limits_{\alpha}^{} P^{n'_\alpha+1} (p^r)$ is an infinite bouquet of mod-$p^r$ Moore spaces, and each $n'_\alpha \geq 4n'-1$; when $p=2$ and $r\geq 2$, by \cite[Lemma 2.6]{Coh89} 
\[
 \Omega P^{n} (2^r)\simeq T^{n}\{2^r\}\times \Omega \Big( \mathop{\bigvee}\limits_{\beta}^{} P^{n'_\beta} (2^r)\Big)
\]
for a space $T^{n}\{2^r\}$, where $\mathop{\bigvee}\limits_{\beta}^{} P^{n'_\beta} (2^r)$ is an infinite bouquet of mod-$2^r$ Moore spaces, and each $n'_\beta > n$; and when $p=2$ and $r=1$, by \cite[Lemma 3.2]{HW20}
\[
\Omega P^{n} (2)\simeq Z\times \mathop{\prod}\limits_{\text{odd}~p} \Omega \Sigma^{\frac{p^2-1}{2}(2n-3)} ( P^{n} (2)\vee  P^{n} (2))
\]
for some space $Z$. 
From these decompositions, it follows that the bottom-cell inclusion $S^{n-2} \stackrel{\widetilde{\jmath}}{\larrow} \Omega P^n(p^r)$ factors through the direct summand $S^{2n'+1}\{p^r\}$, $T^{2n'+1}\{p^r\}$, $T^{n}\{2^r\}$ or $Z$, corresponding to the four respective cases. Accordingly, the cokernel of $\widetilde{\jmath}_\ast$ contains the homotopy groups of Moore spaces $P^\ast(p^r)$ up to isomorphism, and then the lemma follows from Lemma \ref{BHWlemma}.
 \end{proof}
 
 \begin{remark} 
 Lemma \ref{cokerlemma} can be also proved directly by \cite[Theorem 1.3]{Boy24} and a result of Burklund-Senger \cite{Bur22} which shows that the number of the torsion summands in the homotopy groups of a sphere grows at most subexponentially.
 \end{remark}
 
\begin{corollary}\label{cokercoro}
Let $n\geq 3$, $k \geq 1$. 
Consider the homomorphism 
\[
\jmath_\ast: \pi_\ast(S^{n-1})\stackrel{}{\larrow} \pi_\ast(P^n(k))
\]
of homotopy groups induced by the bottom cell inclusion $\jmath$. The following hold: 
\begin{itemize}
 \item if $p^r~|~k$ for an odd prime $p$ and $r\geq 1$, then the numbers of $\mathbb{Z}/p^r$- and $\mathbb{Z}/p^{r+1}$-summands in the cokernel of $\jmath_\ast$ grow exponentially;
 \item if $2^r~|~k$ for $r\geq 2$, then the numbers of $\mathbb{Z}/2^r$- and $\mathbb{Z}/2^{r+1}$-summands in the cokernel of $\jmath_\ast$ grow exponentially; 
  \item if $2~|~k$ but $4~\nmid~k$, then the numbers of $\mathbb{Z}/2$-, $\mathbb{Z}/4$-, and $\mathbb{Z}/8$-summands in the cokernel of $\jmath_\ast$ grow exponentially.
\end{itemize}
\end{corollary}
\begin{proof}
The corollary follows Lemma \ref{cokerlemma} and the fact that $P^n(k)\simeq\mathop{\bigvee}\limits_{j=1}^\ell P^n(p_j^{t_j})$ if $k=\mathop{\prod}\limits_{j=1}^\ell p_j^{t_j}$.
\end{proof}
 
 We apply the special case of Moore spaces to study our capped complexes. 

 \begin{lemma}\label{Znoninert-lemma}
 Let $X$ be the capped complex in Theorem \ref{Zdecom-thm}. If $k\neq \pm 1$, then the attaching map for the top cell of $X$ is not inert. More precisely, the attaching map for the top cell of $X$ is not locally inert after localization at any prime $p$ that divides $k$. 
 
 Moreover, consider the homomorphism $i_\ast: \pi_\ast(\overline{X})\stackrel{}{\larrow}\pi_\ast(X)$ of homotopy groups induced by the lower skeleton inclusion. The following hold:
 \begin{itemize}
  \item if $k=0$, then the number of $\mathbb{Q}$-summands in the cokernel of $i_\ast$ grows exponentially;
 \item if $k=0$, then the number of $\mathbb{Z}/p^r$-summands in the cokernel of $i_\ast$ grows exponentially for all primes $p$ and all $r\geq 1$;
 \item if $k\neq 0$, then the numbers of $\mathbb{Z}/p^r$- and $\mathbb{Z}/p^{r+1}$-summands in the cokernel of $i_\ast$ grow exponentially for all odd primes $p$ with $p^r~|~k$ and $r\geq 1$;
 \item if $k\neq 0$ and $4~|~k$, then the numbers of $\mathbb{Z}/2^r$- and $\mathbb{Z}/2^{r+1}$-summands in the cokernel of $i_\ast$ grow exponentially for all $2^r~|~k$ with $r\geq 1$;
 \item if $k\neq 0$ and $2~|~k$ but $4~\nmid k$, then the numbers of $\mathbb{Z}/2$-, $\mathbb{Z}/4$-, and $\mathbb{Z}/8$-summands in the cokernel of $i_\ast$ grow exponentially.
 \end{itemize}
 \end{lemma}
\begin{proof}
When $k=0$, the homotopy cofibration for $X$ splits and is homotopy equivalent to the homotopy cofibration
\[
S^{n-1}\stackrel{0\mathfrak{g}}{\larrow} S^{m}\vee S^{n-m}\vee C\stackrel{i_{123}}{\larrow} S^{m}\vee S^{n-m}\vee C \vee S^{n},
\]
where $i_{123}$ is the canonical inclusion of the first three wedge summands. 
So the attaching map for the top cell of $X$ is not inert. Moreover, up to isomorphism the cokernel of $i_\ast=i_{123\ast}$ contains the homotopy groups $\pi_{\geq n} (S^n\vee S^m)$. Since $S^n\vee S^m$ is $\mathbb{Z}/p^r$-hyperbolic for all primes $p$ and all $r\geq 1$ by \cite[Theorem 1]{Boy22}, the number of $\mathbb{Z}/p^r$-summands in the cokernel of $i_\ast$ grows exponentially for all primes $p$ and all $r\geq 1$. Also, since $S^n\vee S^m$ is rationally hyperbolic, the number of $\mathbb{Q}$-summands in the cokernel of $i_\ast$ grows exponentially. 

When $k\neq \pm 1, 0$, consider the homotopy cofibration diagram
\[
\diagram
S^{n-1} \dto^{0} \rto^<<<{\ast}   &  S^{m}\vee S^{n-m}\vee C  \ddouble \rto^<<<<{i_{123}}  &  S^{m}\vee S^{n-m}\vee C \vee S^n \dto^{i\circ q_{123}} \\
S^{n-1}            \rto^<<<{k\mathfrak{g}}   &  S^{m}\vee S^{n-m}\vee C   \rto^<<<<<<<<{i}  & X,
\enddiagram
\]
where $q_{123}$ is the pinch map onto the first three wedge summands. Then Theorem \ref{Zdecom-thm} plus the naturality property in Remark \ref{natremark} can apply to deduce the compatible homotopy equivalences
  \[
   \diagram
   \Omega (S^m\times S^{n-m})\times\Omega \big((S^{n-1}\vee S^n \vee C)\rtimes \Omega (S^m\times S^{n-m})\big) \rto^<<<<{\simeq}  \dto^{1\times \Omega ((\jmath\circ q_1)\vee 1)\rtimes 1)} & \Omega (S^{m}\vee S^{n-m}\vee C \vee S^n)\dto^{\Omega (i\circ q_{123}) }\\
      \Omega (S^m\times S^{n-m})\times\Omega \big((P^{n}(k)\vee C)\rtimes \Omega (S^m\times S^{n-m})\big) \rto^<<<<<<<<<{\simeq}   &  \Omega X,
   \enddiagram
 \] 
in which we identify the map $P^n(0)\stackrel{\jmath_0}{\larrow} P^{n}(k)$ with $S^{n-1}\vee S^n \stackrel{ q_1}{\larrow} S^{n-1} \stackrel{\jmath}{\larrow}P^{n}(k)$, where $q_1$ is the pinch map and $\jmath$ is the bottom cell inclusion. 

By the compatible homotopy equivalences and Lemma \ref{semi-lemma}, up to isomorphism the cokernel of $(i\circ q_{123})_\ast$ contains the cokernel of 
\[
((\jmath\circ q_1)\rtimes 1)_\ast: \pi_{\ast}((S^{n-1}\vee S^n)\rtimes \Omega (S^m\times S^{n-m})) \stackrel{}{\larrow} \pi_{\ast}(P^{n}(k))\rtimes \Omega (S^m\times S^{n-m})),
\]
and then by \cite[Proposition 4.6]{Hua24} contains the cokernel of 
\[
(\jmath\circ q_1)_\ast: \pi_{\ast}(S^{n-1}\vee S^n) \stackrel{q_{1\ast}}{\larrow} \pi_\ast(S^{n-1}) \stackrel{\jmath_\ast}{\larrow}
\pi_{\ast}(P^{n}(k)).
\]
Note that the latter agrees with the cokernel of $\jmath_\ast$ as $q_{1\ast}$ is surjective. 
By Corollary \ref{cokercoro}, the numbers of $\mathbb{Z}/p^r$- and $\mathbb{Z}/p^{r+1}$-summands in the cokernel of $(\jmath\circ q_1)_\ast$ grow exponentially for all odd primes $p$ with $p^r~|~k$; the numbers of $\mathbb{Z}/2^r$- and $\mathbb{Z}/2^{r+1}$-summands in the cokernel of $(\jmath\circ q_1)_\ast$ grow exponentially for all $2^r~|~k$ if $4~|~k$; and the numbers of $\mathbb{Z}/2$-, $\mathbb{Z}/4$- and $\mathbb{Z}/8$-summands in the cokernel of $(\jmath\circ q_1)_\ast$ grow exponentially if $2~|~k$ but $4~\nmid~k$. Then the same conclusion holds for the cokernel of $(i\circ q_{123})_\ast$, and hence holds for the cokernel of $i_\ast$ as $q_{123\ast}$ is surjective. In particular, the attaching map $k\mathfrak{g}$ for the top cell of $X$ is not locally inert after localization at any prime $p$ that divides $k$. 
\end{proof}

\begin{proof}[Proof of Theorem \ref{Zinert-thm}]
The theorem follows immediately from Lemmas \ref{Zinert-lemma} and \ref{Znoninert-lemma}.
\end{proof}

\subsection{A refinement, I}
\label{subsec: refine1}
Let $\mathcal{M}$ denote the following class of spaces. 
A space $Y$ belongs to $\mathcal{M}$ if and only if $Y$ is homotopy equivalent to a finite-type wedge of simply connected spheres and Moore spaces.

For a finite abelian group $T\cong \mathop{\bigoplus}\limits_{j=1}^\ell \mathbb{Z}/p_j^{t_j}$, denote by $P^{n}(T)=\mathop{\bigvee}\limits_{j=1}^\ell P^n(p_j^{t_j})$.
 \begin{proposition}\label{ref-prop}
Let $X$ be the capped complex in Theorem \ref{Zdecom-thm}. Write 
\[
H_i(X;\mathbb{Z})\cong \mathbb{Z}^{\oplus d_i}\oplus T_i,
\]
where $d_i\geq 0$ and $T_i$ is a finite abelian group. Suppose that $\overline{X}\in \mathcal{M}$. Then there is a homotopy fibration
\[
(P^{n}(k)\vee C)\rtimes \Omega (S^m\times S^{n-m})\stackrel{}{\larrow} X\stackrel{}{\larrow} S^m\times S^{n-m},
\]
where 
\[
C\simeq \Big(\mathop{\bigvee}\limits_{d_{m}-1} S^{m}\Big)\vee \Big(\mathop{\bigvee}\limits_{d_{n-m}-1} S^{n-m}\Big) \vee \Big(\mathop{\bigvee}\limits_{\substack{i\neq m, n-m \\ 2 \leq i \leq n-1}} \mathop{\bigvee}\limits_{d_i} S^{i} \Big) \vee \Big(\mathop{\bigvee}\limits_{3\leq i\leq n-1}  P^i(T_i) \Big),
\]
and this homotopy fibration splits after looping to give a homotopy equivalence 
   \[
   \Omega X\simeq\Omega (S^m\times S^{n-m})\times\Omega \big((P^{n}(k)\vee C)\rtimes \Omega (S^m\times S^{n-m})\big). 
 \] 
 \end{proposition}
 \begin{proof}
 By assumption $\overline{X}\simeq S^m\vee S^{n-m}\vee C\in \mathcal{M}$. By \cite[Theorem 3.5]{Sta25}, $\mathcal{M}$ is closed under retracts. It follows that $C\in \mathcal{M}$, and then the homotopy equivalence for $C$ follows from the homology of $X$. The other parts of the proposition follows from Theorem \ref{Zdecom-thm} immediately.
 \end{proof}


\section{Poincar\'{e} Duality complexes and their variants}
\label{sec: PD}

Let $M$ be an $n$-dimensional simply connected Poincar\'{e} Duality complex. Then there is a homotopy cofibration
\[
S^{n-1}\stackrel{\mathfrak{g}}{\larrow} \overline{M}\stackrel{}{\larrow} M,
\]
where $\overline{M}$ is the $(n-1)$-skeleton of $M$, and $\mathfrak{g}$ is the attaching map for the top cell of $M$. It is clear that $M$ is a capped complex. 
Furthermore, for any $k\in \mathbb{Z}$, we can define a space $X_{M,k}$ by the homotopy cofibration
\[
S^{n-1}\stackrel{k\mathfrak{g}}{\larrow} \overline{M}\stackrel{}{\larrow} X_{M,k}.
\]
It is clear that $X_{M,k}$ is also a capped complex, and $X_{M,1}\simeq M$.

In this section, we apply our results to the capped complexes $X_{M,k}$ and extend the applications developed in Section \ref{sec: ap} to this setting. In particular, we establish Theorem \ref{QZinert-thm} on the gaps between rational and local or integral inertness and Theorem \ref{Mp-thm} for highly connected $X_{M,k}$.

\subsection{Two decompositions}
\label{subsec: 2decomM}
Theorems \ref{Zdecom-thm} and \ref{pdecom-thm} can be applied to $X_{M,k}$ directly. 

\begin{proposition}\label{MZprop}
Let $M$ be an $n$-dimensional simply connected Poincar\'{e} Duality complex with a homotopy cofibration
\[
S^{n-1}\stackrel{\mathfrak{g}}{\larrow} S^{m}\vee S^{n-m}\vee C\stackrel{}{\larrow} M
\]
for a space $C$ and $2\leq m\leq n-2$. Suppose that the component of $\mathfrak{g}$ on $S^m\vee S^{n-m}$ is the Whitehead product of the canonical inclusions $S^m\stackrel{}{\larrow} S^m\vee S^{n-m}$ and $S^{n-m}\stackrel{}{\larrow} S^m\vee S^{n-m}$. Then for any $k\in \mathbb{Z}$, there is a homotopy fibration
\[
(P^{n}(k)\vee C)\rtimes \Omega (S^m\times S^{n-m})\stackrel{}{\larrow} X_{M,k}\stackrel{}{\larrow} S^m\times S^{n-m},
\]
which splits after looping to give a homotopy equivalence 
   \[
   \Omega X_{M,k}\simeq\Omega (S^m\times S^{n-m})\times\Omega \big((P^{n}(k)\vee C)\rtimes \Omega (S^m\times S^{n-m})\big). 
 \] 
\end{proposition}
\begin{proof}
By assumption, there is a homotopy cofibration
\[
S^{n-1}\stackrel{k\mathfrak{g}}{\larrow}  S^{m}\vee S^{n-m}\vee C\stackrel{}{\larrow} X_{M,k}.
\]
The proposition follows by applying Theorem \ref{Zdecom-thm} to the homotopy cofibration.
\end{proof}

\begin{proposition}\label{Mpprop} 
Let $M$ be an $n$-dimensional simply connected Poincar\'{e} Duality complex such that its $(n-1)$-skeleton is a co-$H$-space. 
Suppose that $(a, b)\in H^{m}(M;\mathbb{Z})\times  H^{n-m}(M;\mathbb{Z})$ with $2\leq m\leq n-2$ is a spherical pair of $M$ such that $\langle a\cup b, [M]\rangle=\pm 1$. Then after localization away from any prime $p$ satisfying $2p<{\rm max}(m, n-m)+4$, for any $k\in \mathbb{Z}\backslash \{0\}$ the following hold:
\begin{itemize}
\item there is a homotopy cofibration
\[
S^{n-1} \larrow S^m\vee S^{n-m}\vee C\stackrel{}{\larrow} X_{M,k}
\]
for a simply connected space $C$ of dimension less than $n-1$;
\item
there is a homotopy fibration
\[
(P^{n}(k)\vee C)\rtimes \Omega (S^m\times S^{n-m})\stackrel{}{\larrow} X_{M,k}\stackrel{}{\larrow} S^m\times S^{n-m};
\]
\item the homotopy fibration splits after looping to give a homotopy equivalence 
   \[
   \Omega X_{M,k}\simeq\Omega (S^m\times S^{n-m})\times\Omega \big((P^{n}(k)\vee C)\rtimes \Omega (S^m\times S^{n-m})\big). 
 \] 
 \end{itemize}
\end{proposition}
\begin{proof}
By construction, there is a homotopy cofibration diagram
\[
\diagram
S^{n-1} \dto^{k} \rto^{k\mathfrak{g}}  & \overline{M} \ddouble \rto^{}  &  X_{M,k} \dto^{\varphi_k} \\
S^{n-1}  \dto^{}  \rto^{\mathfrak{g}}  & \overline{M}   \rto^{}         \dto^{}  &  M\dto^{}\\
P^n(r)  \rto^{}                                  &  \ast            \rto^{}                 & P^{n+1}(r), 
\enddiagram
\] 
where the induced map $\varphi_k$ is of degree $k$. 
It follows that the $(n-1)$-skeleton of $X_{M,k}$ is a co-H-space, 
\[
a^2=0, \ \ \ b^2=0,  \ \ \  {\rm and} \ \ \  \langle a\cup b, [X_{M,k}]\rangle=k\langle a\cup b, [M]\rangle=\pm k,
\]
where $a=\varphi_k^\ast (a)\in H^{m}(X_{M,k};\mathbb{Z})$ and $b=\varphi_k^\ast (b)\in H^{n-m}(X_{M,k};\mathbb{Z})$ by abuse of notation. 

For the homotopy cofibration
\[
X_{M,k}\stackrel{\varphi_k}{\larrow} M \stackrel{}{\larrow} P^{n+1}(r),
\] 
the Blakers-Massey Theorem implies that it is a homotopy fibration up to degree $n$. In particular, the map $X_{M,k}\stackrel{\varphi_k}{\larrow}  M$ is $(n-1)$-connected. Since $(a, b)\in H^{m}(M;\mathbb{Z})\times  H^{n-m}(M;\mathbb{Z})$ is a spherical pair, there exist maps
\[
s_1: S^m\larrow M \ \ \ {\rm and} \ \ \  s_2: S^{n-m}\larrow M
\]
such that $s_1^\ast(a)$ and $s_2^\ast(b)$ are generators of $H^{m}(S^m;\mathbb{Z})$ and $H^{n-m}(S^{n-m};\mathbb{Z})$, respectively. By assumption $m$, $n-m\leq n-2$, it follows that the maps $s_1$ and $s_2$ can be lifted through $X_{M,k}\stackrel{\varphi_k}{\larrow} M$ to maps
\[
s_{1,k}: S^m\larrow X_{M,k} \ \ \ {\rm and} \ \ \  s_{2,k}: S^{n-m}\larrow X_{M,k},
\]
and then $s_{1,k}^{\ast}(a)$ and $s_{2,k}^{\ast}(b)$ are generators of $H^{m}(S^m;\mathbb{Z})$ and $H^{n-m}(S^{n-m};\mathbb{Z})$, respectively. 

To summarize, we have showed that the $(n-1)$-skeleton of $X_{M,k}$ is a co-H-space, the attaching map for the top cell of $X_{M,k}$ is divisible by $k$, and $(a, b)\in H^{m}(X_{M,k};\mathbb{Z})\times  H^{n-m}(X_{M,k};\mathbb{Z})$ is a spherical pair with $\langle a\cup b, [X_{M,k}]\rangle=\pm k$. Hence, we can apply Theorem \ref{pdecom-thm} to $X_{M,k}$ and the proposition follows. 
\end{proof}

\subsection{Inertness and non-inertness, II}
\label{subsec: inert2}
We continue our study of inertness and non-inertness in Subsection \ref{subsec: inert1} for $X_{M,k}$ and prove Theorem \ref{QZinert-thm}. 

\begin{proposition}\label{QZinert-prop}
Let $M$ be an $n$-dimensional simply connected Poincar\'{e} Duality complex with a homotopy cofibration
\[
S^{n-1}\stackrel{\mathfrak{g}}{\larrow} S^{m}\vee S^{n-m}\vee C\stackrel{}{\larrow} M
\]
for a space $C$ and $2\leq m\leq n-2$. Suppose that the component of $\mathfrak{g}$ on $S^m\vee S^{n-m}$ is the Whitehead product of the canonical inclusions $S^m\stackrel{}{\larrow} S^m\vee S^{n-m}$ and $S^{n-m}\stackrel{}{\larrow} S^m\vee S^{n-m}$. 
Then the attaching map for the top cell of $X_{M,k}$ is inert if and only if $k=\pm 1$.

In particular, when $k\neq \pm 1, 0$, the attaching map for the top cell of $X_{M,k}$ is rationally inert but not integrally inert, and not locally inert after localization at any prime $p$ that divides $k$.  
\end{proposition}
\begin{proof}
We have seen that $X_{M,k}$ satisfies the condition of Theorem \ref{Zdecom-thm}. So the first statement follows from Theorem \ref{Zinert-thm}. 

For the second statement, by Lemma \ref{Znoninert-lemma} we only need to show that $X_{M,k}$ is rationally inert when $k\neq \pm 1, 0$. To show this, note that $X_{M,k}$ is rationally homotopy equivalent to $X_{M,1}\simeq M$, and then is a rational Poincar\'{e} Duality complex. By a classical result of Halperin-Lemaire \cite{HL87}, the attaching map for the top cell of a rational Poincar\'{e} Duality complex is rationally inert unless its rational cohomology algebra is generated by a single element. However, by Diagram \eqref{CXQdiag} there is a degree one map $M\stackrel{r}{\larrow} S^m\times S^{n-m}$ restricting to the pinch map on the $(n-1)$-skeletons. In particular, $r^\ast: H^\ast(M;\mathbb{Q})\stackrel{}{\larrow} H^\ast(S^m\times S^{n-m};\mathbb{Q})$ is surjective, and then the cohomology algebra $H^\ast(X_{M,k};\mathbb{Q})\cong H^\ast(M;\mathbb{Q})$ is not generated by a single element. Accordingly, the attaching map for the top cell of $X_{M,k}$ is rationally inert. 
\end{proof}

\begin{proof}[Proof of Theorem \ref{QZinert-thm}]
Let $M(n,m,t)=\mathop{\sharp}\limits_{t} ~(S^m\times S^{n-m})$ for each positive integer $t$, and $2\leq m\leq n-2$. Then $X_{M(n,m,t), k}\simeq X_{M(n',m',t'), k'}$ if and only if $(n,m,t, k)= (n',m',t', k')$. The complex $X_{M(n,m,t), k}$ is a capped complex, and is a rational Poincar\'{e} Duality complex if and only if $k\neq 0$. By the result of Halperin-Lemaire \cite{HL87}, any such $X_{M(n,m,t), k}$ is rationally inert. 

When $k\neq \pm 1, 0$, each $X_{M(n,m,t), k}$ satisfies the condition of Proposition \ref{QZinert-prop}, and hence its top-cell attaching map is rationally inert but not integrally inert. Further, when $p~|~k$ for a prime $p$, Proposition \ref{QZinert-prop} again implies that the top-cell attaching map of each $X_{M(n,m,t), k}$ is rationally inert but not locally inert at $p$. 
\end{proof}

\begin{remark}\label{N-remark}
Let $M=M(n,m,t)\sharp N$ with $N$ an $n$-dimensional simply connected Poincar\'{e} Duality complex. By the same proof as Theorem \ref{QZinert-thm}, it can be shown that the top cell attachment of $X_{M,k}$ is rationally inert if $k\neq 0$, is not integrally inert if $k\neq \pm 1, 0$, and is not locally inert at $p$ if $k\neq 0$ and $p~|~k$.  
\end{remark}

\begin{remark}\label{ex-remark}
Note that each complex $M=M(n,m,t)\sharp N$ in Remark \ref{N-remark} satisfies the condition of Theorem \ref{Zdecom-thm}. 
By Lemma \ref{Znoninert-lemma}, if $k\neq 0$ then the numbers of $\mathbb{Z}/p^r$- and $\mathbb{Z}/p^{r+1}$-summands in the cokernel of $\pi_\ast(\overline{X_{M, k}})\stackrel{i_\ast}{\larrow} \pi_\ast(X_{M, k})$ grow exponentially for all odd primes $p$ with $p^r~|~k$ and $r\geq 1$; if $k\neq 0$ and $4~|~k$, then the numbers of $\mathbb{Z}/2^r$- and $\mathbb{Z}/2^{r+1}$-summands in the cokernel of $i_\ast$ grow exponentially for all $2^r~|~k$ with $r\geq 1$; 
 and if $k\neq 0$ and $2~|~k$ but $4~\nmid k$, then the numbers of $\mathbb{Z}/2$-, $\mathbb{Z}/4$-, and $\mathbb{Z}/8$-summands in the cokernel of $i_\ast$ grow exponentially.

Additionally, if $k\neq 0$ and $t\geq 2$, then each $X_{M, k}$ is also rationally hyperbolic. Indeed, by Theorem \ref{Zdecom-thm} there is a rational homotopy equivalence
 \[
   \Omega X_{M, k}\simeq\Omega (S^m\times S^{n-m})\times\Omega \big(C\rtimes \Omega (S^m\times S^{n-m})\big),
 \] 
where $C\simeq  \mathop{\bigvee}\limits_{t-1} ~(S^m\vee S^{n-m})\vee \overline{N}$. By Lemma \ref{semi-lemma}, $(S^m\vee S^{n-m}) \rtimes \Omega (S^m\times S^{n-m})$ retracts off $C\rtimes \Omega (S^m\times S^{n-m})$, and therefore $\Omega (S^m\vee S^{n-m})$ retracts off $\Omega X_{M, k}$. Then since $S^m\vee S^{n-m}$ is rationally hyperbolic, so is $X_{M, k}$. 
\end{remark}

\subsection{A refinement, II}
\label{subsec: refine2}
We continue our study in Subsection \ref{subsec: refine1} for $X_{M,k}$, and prove Theorem \ref{Mp-thm} together with its analogue, Theorem \ref{Wp-thm}.   
Recall $\mathcal{M}$ is the class of spaces homotopy equivalent to a finite-type wedge of simply connected spheres and Moore spaces. Note that $\overline{X_{M, k}}\simeq \overline{M}$.  
 
 \begin{lemma}\label{pMlemma}
Let $M$ be an $(l-1)$-connected Poincar\'{e} Duality complex of dimension $n\leq 3l-2$ with $l\geq 3$. Then after localization away from any prime $p$ satisfying $2p\leq n-2l+3$, $\overline{X_{M, k}}\simeq \overline{M}\in \mathcal{M}$. 
\end{lemma}
\begin{proof}
By Poincar\'{e} Duality and the assumption, the $(n-1)$-skeleton $\overline{M}$ of $M$ is $(l-1)$-connected and of dimension $n-l$. Since $n-l\leq 2l-2$ and $l\geq 3$, by \cite[Lemma 7.8]{ST25}, $\overline{X_{M, k}}\simeq \overline{M}\in \mathcal{M}$ after localization away from any prime $p$ satisfying $2p\leq (n-l)-l+3$. 
\end{proof}

 \begin{proof}[Proof of Theorem \ref{Mp-thm}]
Localize away from any prime $p$ satisfying $2p<{\rm max}(m, n-m)+4$. 
 Note that 
 \[
 ({\rm max}(m, n-m)+4)-(n-2l+4)={\rm max}(m-n+2l, 2l-m)\geq {\rm max}(m-l+2, 3l-n)\geq 2.
 \] 
 Then Lemma \ref{pMlemma} can apply to show that $\overline{X_{M, k}}\simeq \overline{M}\in \mathcal{M}$. Since $\langle a\cup b, [M]\rangle=\pm 1$, the classes $a$ and $b$ are generators of infinite order, and then there exist maps 
\[
s_1: S^m\larrow \overline{M} \larrow M \ \ \ {\rm and} \ \ \  s_2: S^{n-m} \larrow \overline{M}\larrow M
\]
such that 
$s_1^\ast(a)$ and $s_2^\ast(b)$ are generators of $H^{m}(S^m;\mathbb{Z})$ and $H^{n-m}(S^{n-m};\mathbb{Z})$, respectively. Combining with the assumption that $a^2=0$ and $b^2=0$, it follows that $(a, b)$ is a spherical pair of $M$. Then the theorem follows from Proposition \ref{Mpprop}, together with the argument used in the proof of Proposition \ref{ref-prop} to identify the homotopy equivalence for $C$.
 \end{proof}
  
A result analogous to Theorem \ref{Mp-thm} can be obtained after killing the torsion in the homology of $M$. 
Let $\mathcal{W}$ denote the following class of spaces. 
A space $Y$ belongs to $\mathcal{W}$ if and only if $Y$ is homotopy equivalent to a finite-type wedge of simply connected spheres.

\begin{lemma}\label{pWlemma}
Let $M$ be an $(l-1)$-connected Poincar\'{e} Duality complex of dimension $n\leq 3l-1$ with $l\geq 2$. Then after localization away from any prime $p$ appearing as $p$-torsion of $H_\ast(M;\mathbb{Z})$ and any prime $p$ satisfying $2p\leq n-2l+3$, $\overline{X_{M, k}}\simeq \overline{M}\in \mathcal{W}$. 
\end{lemma}
\begin{proof}
By Poincar\'{e} Duality and the assumption, the $(n-1)$-skeleton $\overline{M}$ of $M$ is $(l-1)$-connected and of dimension $n-l$. Since $n-l\leq 2l-1$ and $l\geq 2$, by \cite[Lemma 7.7]{ST25}, $\overline{X_{M, k}}\simeq \overline{M}\in \mathcal{W}$ after localization away from any prime $p$ appearing as $p$-torsion of $H_\ast(\overline{M};\mathbb{Z})$ and any prime $p$ satisfying $2p\leq (n-l)-l+3$.
\end{proof}

 \begin{theorem}\label{Wp-thm}
 Let $M$ be an $(l-1)$-connected Poincar\'{e} Duality complex of dimension $n\leq 3l-1$ with $l\geq 2$. Let $d_i$ be the $i$-th Betti number of $M$. 
 Suppose that there exists $(a, b)\in H^{m}(M;\mathbb{Z})\times  H^{n-m}(M;\mathbb{Z})$ such that $l\leq m\leq n-l$, $a^2=0$, $b^2=0$ and $\langle a\cup b, [M]\rangle=\pm 1$. Then after localization away from any prime $p$ appearing as $p$-torsion of $H_\ast(M;\mathbb{Z})$ and any prime $p$ satisfying $2p<{\rm max}(m, n-m)+4$, for any $k\in \mathbb{Z}\backslash \{0\}$ the following hold:
\begin{itemize}
\item there is a homotopy cofibration
\[
S^{n-1} \larrow S^m\vee S^{n-m}\vee C\stackrel{}{\larrow} X_{M,k},
\]
where
\[
C\simeq \Big(\mathop{\bigvee}\limits_{d_{m}-1} S^{m}\Big)\vee \Big(\mathop{\bigvee}\limits_{d_{n-m}-1} S^{n-m}\Big) \vee \Big(\mathop{\bigvee}\limits_{\substack{i\neq m, n-m \\ l \leq i \leq n-l}} \mathop{\bigvee}\limits_{d_i} S^{i} \Big); 
\]
\item
there is a homotopy fibration
\[
(P^{n}(k)\vee C)\rtimes \Omega (S^m\times S^{n-m})\stackrel{}{\larrow} X_{M,k}\stackrel{}{\larrow} S^m\times S^{n-m};
\]
\item the homotopy fibration splits after looping to give a homotopy equivalence 
   \[
   \Omega X_{M,k}\simeq\Omega (S^m\times S^{n-m})\times\Omega \big((P^{n}(k)\vee C)\rtimes \Omega (S^m\times S^{n-m})\big). 
 \] 
 \end{itemize}
  \end{theorem}
 \begin{proof}
Localize away from any prime $p$ appearing as $p$-torsion of $H_\ast(M;\mathbb{Z})$ and any prime $p$ satisfying $2p<{\rm max}(m, n-m)+4$.  
 Note that 
 \[
 ({\rm max}(m, n-m)+4)-(n-2l+4)={\rm max}(m-n+2l, 2l-m)\geq {\rm max}(m-l+1, 3l-n)\geq 1.
 \] 
 Then Lemma \ref{pWlemma} can apply to show that $\overline{X_{M, k}}\simeq \overline{M}\in \mathcal{W}$. Since $\langle a\cup b, [M]\rangle=\pm 1$, the classes $a$ and $b$ are generators of infinite order, and then there exist maps 
\[
s_1: S^m\larrow \overline{M} \larrow M \ \ \ {\rm and} \ \ \  s_2: S^{n-m} \larrow \overline{M}\larrow M
\]
such that 
$s_1^\ast(a)$ and $s_2^\ast(b)$ are generators of $H^{m}(S^m;\mathbb{Z})$ and $H^{n-m}(S^{n-m};\mathbb{Z})$, respectively. Combining with the assumption that $a^2=0$ and $b^2=0$, it follows that $(a, b)$ is a spherical pair of $M$. Then the theorem follows from Proposition \ref{Mpprop}, together with the argument used in the proof of Proposition \ref{ref-prop} to identify the homotopy equivalence for $C$.
 \end{proof}
 
When $k=1$, Stanton-Theriault \cite[Theorem 1.4]{ST25} and Basu-Basu \cite{BB19} proved loop space decompositions by different methods. In particular, Theorem \ref{Wp-thm} can be viewed as a generalization of their results. 
 
\begin{remark}
Localize away from any prime $p$ appearing as $p$-torsion of $H_\ast(M;\mathbb{Z})$ and any prime $p$ satisfying $2p<{\rm max}(m, n-m)+4$. 
For the complex $X_{M,k}$ with $k\neq 0$ in Theorem \ref{Wp-thm}, by the same argument used in the proof of \cite[Theorem 8.4]{ST25}, we can deduce from Theorem \ref{loopH-thm} that there is an isomorphism of Hopf algebras
\[
H_\ast(\Omega X_{M,k}) \cong TV/(kI) \cong T\Big(\mathop{\bigoplus}\limits_{i=l}^{n-l} V_i\Big)/(kI), 
\]
where the free local module $V_i$ is of degree $i$ and of rank $d_i$, and $I$ is a sum of monomials in the generators of $V$. Moreover, if $n\geq 3l-2$, $I$ is quadratic by \cite[Remark 8.5]{ST25}.
\end{remark}

\begin{remark}\label{pair-remark}
In Theorems \ref{Mp-thm} and \ref{Wp-thm}, suppose that $M$ satisfies one of the following:
\begin{itemize}
\item $n$ is even and $H_{\rm odd}(M;\mathbb{Q})$ is not trivial; or 
\item $n$ is odd and $H_{\rm odd<\frac{n+1}{2}}(M;\mathbb{Q})$ is not trivial.
\end{itemize}
Then by Ponicar\'{e} Duality there exists a pair $(a,b)$ satisfying the hypothesis of theorem. 
 \end{remark}


\bibliographystyle{amsalpha}

\end{document}